\documentclass[11pt,a4paper]{amsart}
\usepackage{amssymb,mathtools}
\usepackage{enumitem,needspace,booktabs,longtable,array}
\usepackage[unicode,colorlinks=true,linkcolor=blue,citecolor=blue,urlcolor=blue,breaklinks=true]{hyperref}
\hypersetup{pdftitle={Frame Flow Ergodicity under Quarter Pinching},pdfauthor={Heng Zhang and Shuhao Zhang}}
\newtheorem{theorem}{Theorem}[section]
\newtheorem{proposition}[theorem]{Proposition}
\newtheorem{lemma}[theorem]{Lemma}
\newtheorem{corollary}[theorem]{Corollary}
\newtheorem{conjecture}[theorem]{Conjecture}
\theoremstyle{definition}

\theoremstyle{remark}

\newcommand{\SM}{SM}
\newcommand{\FM}{FM}
\newcommand{\R}{\mathbb R}\newcommand{\C}{\mathbb C}\newcommand{\Z}{\mathbb Z}

\DeclareMathOperator{\Sym}{Sym}\DeclareMathOperator{\Id}{Id}
\DeclareMathOperator{\tr}{tr}\DeclareMathOperator{\im}{im}\DeclareMathOperator{\rank}{rank}
\DeclareMathOperator{\End}{End}\DeclareMathOperator{\vol}{vol}
\DeclareMathOperator{\Sec}{sec}\DeclareMathOperator{\Ric}{Ric}\DeclareMathOperator{\scal}{scal}
\DeclareMathOperator{\Frac}{Frac}
\DeclareMathOperator{\SO}{SO}
\DeclareMathOperator{\Spin}{Spin}

\numberwithin{equation}{section}
\allowdisplaybreaks[1]
\title{Pinching and Tensorial Rigidity for Ergodicity of Frame Flows}
\author{Heng Zhang}
\address{School of Mathematical Sciences, University of Science and Technology of China, Hefei, China}
\email{hengz@mail.ustc.edu.cn}

\author{Shuhao Zhang}
\address{School of Mathematical Sciences, University of Science and Technology of China, Hefei, China}
\email{yichen12@mail.ustc.edu.cn}

\date{\today}
\begin{document}
\begin{abstract}
Let $(M^n,g)$ be a closed, connected, oriented Riemannian manifold with negative sectional curvature and $n\ge 2$. It is known, due to Brin and Gromov, that the oriented frame flow is ergodic in every odd dimension $n\ge 3$, $n\neq 7$. We prove that if a closed, connected, oriented negatively curved Riemannian manifold is strictly quarter-pinched, then its oriented frame flow is ergodic. 
\end{abstract}
\maketitle

\section{Introduction}
\label{main:introduction}
Let $(M^n,g)$, $n\geq3$, be a smooth closed connected oriented Riemannian manifold with
negative sectional curvature. We write
$$
  \pi:\SM\longrightarrow M,
  \qquad
  \SM=\{(x,v):x\in M,\ v\in T_xM,\ |v|_g=1\}
$$
for its unit tangent bundle, and denote by $\FM\to M$ the oriented
orthonormal frame bundle.  A point of $\FM$ over $x$ will be represented by
an oriented orthonormal basis
$
  (v,e_2,\ldots,e_n)
$
of $T_xM$.  Retaining only the first vector defines a principal bundle
$$
  p:\FM\longrightarrow\SM,
  \qquad
  p(v,e_2,\ldots,e_n)=(x,v),
$$
with structure group $\SO(n-1)$.

Let $(\varphi_t)_{t\in\R}$ be the geodesic flow on $\SM$.  The oriented
frame flow $(\Phi_t)_{t\in\R}$ is obtained by moving $(x,v)$ under
$\varphi_t$ and parallel-transporting the remaining vectors along the same
geodesic.  Thus
$
  p\circ\Phi_t=\varphi_t\circ p.
$
The flow preserves the smooth probability measure formed from Liouville
measure on $\SM$ and Haar measure on the $\SO(n-1)$-fibres, and ergodicity
will always refer to this measure.  Because negative curvature makes the
geodesic flow Anosov, the frame flow is a compact isometric extension of an
Anosov flow and is one of the basic examples of partial hyperbolicity; see
\cite{HasselblattPesin06}.

Our starting point is a conjecture of Brin
\cite[Conjecture~2.6]{Brin82} relating curvature pinching to the
ergodicity of the frame flow. To state it, recall that $g$ is
\emph{strictly quarter-pinched} if the sectional curvature
satisfies the uniform bounds
$$
  -\kappa \leq \Sec_g \leq -\delta\kappa
$$
for some constants $\kappa>0$ and $\delta>1/4$.

\begin{conjecture}[Brin]\label{conj:Brin}
If a closed, connected, oriented negatively curved Riemannian manifold is strictly
quarter-pinched, then its oriented frame flow is ergodic.
\end{conjecture}

The value $1/4$ is forced by geometry rather than by the known estimates.
Indeed, a parallel complex structure produces a proper invariant reduction
of the oriented frame bundle, so a negatively curved K\"ahler manifold does
not have an ergodic frame flow on the oriented orthonormal frame bundle.  Berger's pinching theorem places such
metrics at or below the quarter-pinched threshold, and complex hyperbolic
metrics attain the pinching constant $1/4$ \cite{Berger60}.  

Brin's original formulation is stronger: under the same curvature
assumption he predicts that the frame flow is Bernoulli, and he further
conjectures ergodicity and the Bernoulli property whenever the holonomy group
is $\SO(n)$ \cite[Conjectures~2.6 and 2.9]{Brin82}.  The Bernoulli conclusion
is substantially stronger than ergodicity; for smooth measure-preserving diffeomorphisms, a recent mechanism by
which exponential mixing yields the Bernoulli property is established in
\cite{DolgopyatKanigowskiRodriguezHertz24}.

The unpinched problem is already settled in most odd dimensions.  Brin and
Gromov proved ergodicity for every odd $n\neq7$ without any curvature
pinching beyond negativity \cite{BrinGromov80}.  Brin also showed that
ergodicity holds on an open dense set of negatively curved metrics in the
$C^3$ topology \cite[Section~5]{Brin82}, although a uniform theorem in even
dimensions, and the expected unpinched statement in dimension $7$, are
considerably subtler.  The first quantitative results required curvature
close to constant: Brin--Karcher obtained the pinching threshold
$0.8649\ldots$ in even dimensions other than $8$, while
Burns--Pollicott obtained $0.9805\ldots$ in dimensions $7$ and $8$
\cite{BrinKarcher84,BurnsPollicott03}.

A different and innovative framework was introduced by
Ceki\'c--Lefeuvre--Moroianu--Semmelmann \cite{CLMS24}.  In the dimensions
relevant to the present paper, their estimates give the thresholds
$0.2928\ldots$ for $n=4$ and $0.2823\ldots$ for $n=6$.  For
$n\equiv2\pmod4$, $n\geq10$ and $n\ne134$, they obtain an explicit increasing sequence of
bounds beginning with $0.2725\ldots$ in dimension 10 and converging to
$0.2779\ldots$.  Their theorem also treats the complementary cases: the
bounds in dimensions divisible by $4$ decrease from $0.5948\ldots$ in
dimension 12 toward $0.5572\ldots$, and the exceptional dimensions
$7$, $8$, and $134$ are covered by the respective thresholds
$0.4962\ldots$, $0.6212\ldots$, and $0.5788\ldots$.  The conceptual
change is as important as the numerical improvement: the earlier arguments
were either essentially topological or based on the geometry of the
universal cover, whereas the method of \cite{CLMS24} converts
non-ergodicity into invariant tensorial data on $\SM$ and then applies a
twisted Pestov identity \cite{GPSU16}.

The ergodicity question is part of a broader theory of compact isometric
extensions of hyperbolic flows.  For frame flows in dimensions $n\geq4$, ergodicity implies mixing
\cite{Lefeuvre23}. Moreover, since the structure group
$\SO(n-1)$ is compact semisimple in this range, ergodicity is
equivalent to rapid mixing by
\cite{CekicLefeuvreSemiclassical24}. Related quantitative criteria,
applying in particular to some frame flows of negatively curved
manifolds, are established in \cite{PollicottZhang25}.  When a K\"ahler
structure obstructs ergodicity of the frame flow on the oriented orthonormal frame bundle, the corresponding
unitary frame flow becomes the natural object: if the complex dimension is
$m$, Brin--Gromov's argument gives ergodicity for $m$ odd or $m=2$, while
holomorphic-pinching results cover even $m\neq4,28$
\cite{BrinGromov80,CLMSUnitary24}.  There is also an analogous theory for an
arbitrary Euclidean vector bundle $(E,\nabla)\to M$ with an orthogonal
connection: if $\operatorname{rank}E=r\leq\sqrt n$ and the holonomy group is
$\SO(r)$, then the associated frame flow is ergodic
\cite{CekicLefeuvreConnections25}.  This gives further evidence for Brin's
holonomy conjecture and indicates that the methods extend well beyond the
tangent bundle.  For an overview of this circle of ideas, see
\cite{CLMS22Survey}.

In this paper, we prove Conjecture \ref{conj:Brin}.

\begin{theorem}\label{main:quarter}
Conjecture \ref{conj:Brin} holds true.
\end{theorem}

\medskip
\noindent\textbf{The strategy of the proof.}
We argue by contradiction and use the geometric reduction of
\cite{CLMS24}. If the frame flow is not ergodic, its transitivity group
is a proper subgroup of $\SO(n-1)$ and determines a flow-invariant
reduction of the frame bundle. The topology of the sphere fibres
restricts the possible structure groups. The correspondence between
fixed vectors of their representations and flow-invariant sections then
produces geometric structures on the unit tangent bundle. After passing
to a finite Riemannian cover, these are unit vector fields and orthogonal
projectors on the normal bundle, together with an orthogonal almost complex
structure in dimension 7 and a $G_2$-structure in dimension 8.
The additional reduction to the adjoint representation in dimension
$134$ is treated separately in Lemma~\ref{main:adjoint134}. The invariant
sections have finite Fourier content: on each sphere fibre, they can be
represented by polynomials in the unit tangent vector. It therefore
remains to exclude these flow-invariant structures under the pinching
assumption.

We first consider even flow-invariant orthogonal projectors on the
normal bundle in dimensions at least 8, without a restriction on
their rank. Their highest spherical harmonic components correspond to
normal twisted conformal Killing tensors in the sense of
\cite[Definition~4.3]{CLMS24}. The normal condition in that definition
controls the degrees of contractions with the unit tangent vector; it
does not require those contractions to vanish. Thus the highest
component need not itself take values in the symmetric endomorphisms of
the normal bundle. We decompose it into a trace-free normal symmetric
tensor component, a normal vector component, and a scalar component.
Representation theory separates the relevant components of their
covariant derivatives. Applying the corresponding differential
estimates and the twisted Pestov identity successively gives
nonpositive integrated curvature expressions, whereas our curvature
estimates give strictly positive lower bounds for nonzero components.
For degree 6, the general high-degree estimate is insufficient, and
we instead retain the relations between the different symmetry components
of the twisting curvature term and estimate them together. This eliminates every degree at least 6 and
leaves a projector of degree at most 4; see Proposition~\ref{pr:degree-bound}.

We exclude degree 4 by an algebraic argument. For the homogeneous
rational extension of a nonzero even polynomial normal projector, we
prove that the least power of the squared norm needed to clear the
denominators has an odd exponent. Consequently, the polynomial degree
of such a projector is congruent to 2 modulo 4. A nonzero projector
of degree at most 4 must therefore have degree 2. We also prove
that its quadratic coefficients depend smoothly on the base point,
including at points where the degree drops. 

For quadratic projectors, we use the algebraic identities expressing
orthogonality, idempotence, and normality. Passing to the orthogonal
complement within the normal bundle and restricting the quadratic
polynomial to 2-dimensional subspaces gives bounds for the derivative
blocks between the image and its orthogonal complement. Integration of
4th-order monomials on the sphere then rewrites the curvature term
in a form that can be estimated using sectional curvature bounds. The
resulting lower bound is strictly positive under strict quarter
pinching whenever the projector has rank other than 0 or full rank
on the normal bundle. This contradicts the integrated twisted Pestov
identity. The quadratic estimate gives the quarter-pinching threshold.

Rank-1 projectors admit a separate argument, valid also in dimensions
4 and 6. We express such a projector as the tensor square of an
odd polynomial unit vector. This factor is unique up to sign, and the
sign ambiguity can be removed on a finite cover of the base. By working
in one fixed coefficient space, we obtain smooth factors even where
their minimal degree changes. We decompose the highest spherical
harmonic into normal vector and scalar components and apply the
corresponding estimates repeatedly. The factor is eventually reduced
to degree 1, so its tensor square is a quadratic projector and is
excluded by the preceding estimate. Applying this argument to the
tensor square of an invariant unit vector field also excludes the
unit-vector-field cases of the geometric reduction.

The almost complex structure in dimension 7 and the $G_2$-structure in
dimension 8 are treated using spinors. We associate to each
structure orthogonal projectors on the spinor bundles and construct
smooth polynomial spinor factors. These give pairs related by Clifford
multiplication, with opposite antipodal parities and Fourier degrees
differing by 1. In dimension 7, differentiation along the
geodesic flow rotates the pair with a real scalar coefficient depending
only on the base point. This coefficient does not contribute to the
highest-degree equation of the factor of larger degree, but remains in
the lower-degree equations. The highest equation yields a nonzero
monogenic spinor polynomial, that is, a polynomial annihilated by the
Dirac operator in the fibre variable, for which the Cartan component of
the covariant derivative induced by the spin connection vanishes.
Weitzenb\"ock and Lichnerowicz identities, combined with curvature
estimates, exclude all positive degrees. Explicit rational estimates
treat the finitely many intermediate degrees, and uniform estimates
treat all sufficiently large degrees. In degree 0, the equations
would give a parallel spinor in dimension 8 or the real Killing
spinor equation with an initially variable coefficient in dimension
7. The Lichnerowicz identity rules out both possibilities under
negative scalar curvature.

Combining these results with the theorem of Brin and Gromov in the
remaining odd dimensions proves the main theorem.

\medskip
\noindent\textbf{Organisation of the article.}
In Section~\ref{main:estimates} we recall the geometric reduction to
flow-invariant sections with finite Fourier content, fix the polynomial
and curvature conventions, and establish estimates that reduce even
flow-invariant projectors to degree at most 4.
Section~\ref{pr:algebraic-section} proves the parity restriction on the
polynomial degree and the quadratic curvature estimate.
In Section~\ref{pr:exclusion-section} we complete the smooth reduction to degree
2 and treat rank-1 projectors of arbitrary finite Fourier degree.
Section~\ref{spin:section} treats the exceptional invariant structures in
dimensions 7 and 8 using spinors. Finally,
Section~\ref{main:assembly} excludes the remaining reduction to the adjoint
representation and completes the proof.

\section{Geometric reduction and harmonic estimates}
\label{main:estimates}
\label{pr:section}

In this section we recall the invariant structures obtained from
non-ergodicity of the frame flow and establish the estimates needed for
their highest Fourier components. The main result is
Proposition~\ref{pr:degree-bound}.

\subsection{Invariant structures}
\label{main:reduction}

Let $N\to\SM$ denote the normal bundle,
$N_{(x,v)}=v^\perp\subset T_xM$. We write $X$ for the generator of
the geodesic flow $\varphi_t$. For a section $f$ of a pullback bundle,
$Xf$ denotes covariant differentiation along this vector field using
the connection on its coefficient bundle. As in
\cite[Section~2.2]{CLMS24}, we write $f=\sum_{k\geq0}f_k$ for its
spherical harmonic decomposition in $v$, with $f_k$ of degree $k$.
The largest nonzero
harmonic degree is called the \emph{Fourier degree}; a section has
\emph{finite Fourier content} when only finitely many components occur.
For a polynomial section on a sphere, this is the least degree of a
polynomial representing it. A homogeneous extension may have a larger
degree.

A section $f$ is \emph{flow-invariant} if $Xf=0$. It is even or odd if,
respectively, $f(x,-v)=f(x,v)$ or $f(x,-v)=-f(x,v)$. We view an
orthogonal projector $P\in\Sym^2N$ as a symmetric endomorphism of
$\pi^*TM$, extended by 0 on the line spanned by $v$. Thus
\begin{equation}\label{eq:geo-projector-identities}
P^2=P,\qquad P^*=P,\qquad P(x,v)v=0.
\end{equation}
Likewise, an orthogonal almost complex structure on $N$ is understood as a
metric-compatible endomorphism of $N$ whose square is $-\Id$, and a
$G_2$-structure on $N$ is understood to induce its prescribed metric.

In dimension 2 an oriented orthonormal frame is determined by its
1st vector. The frame flow is therefore the geodesic flow and is
ergodic in negative curvature; see
\cite[Remark~1.2(a)]{BrinGromov80}. In odd dimensions different from
7, ergodicity follows from \cite[Theorem~1.1]{BrinGromov80}.
In the remaining dimensions we use the following consequence of
\cite[Theorem~3.8]{CLMS24}.

Below $\rho(n)$ is the Radon--Hurwitz number: if
$n=(2a+1)2^{b+4c}$ with $0\leq b\leq3$, then
$\rho(n)=2^b+8c$.

\begin{proposition}\label{main:alternatives}
Let $(M^n,g)$ be a closed connected oriented negatively curved
Riemannian manifold, and suppose that its frame flow is not ergodic.
After passing to a finite Riemannian cover, the following conclusions
hold.
\begin{enumerate}[label=\textup{(\roman*)}]
\item If $n=4$, or $n\equiv2\pmod4$ and $n\ne134$, there is an odd
flow-invariant unit section of $N$ with finite Fourier content.
\item If $4\mid n$ and $n\geq12$, there is an even flow-invariant
orthogonal projector in $\Sym^2N$, with finite Fourier content and rank
$$
1\leq r\leq\min\{\rho(n)-1,(n-2)/2\}.
$$
\item If $n=7$, there is an odd flow-invariant orthogonal almost complex
structure on $N$ with finite Fourier content.
\item If $n=8$, there is either an even flow-invariant orthogonal
projector in $\Sym^2N$ of rank 1, 2 or 3, or an odd
flow-invariant $G_2$-structure on $N$. Each has finite Fourier content.
\item If $n=134$, there is either an odd flow-invariant unit section of
$N$ with finite Fourier content, or a reduction of the structure group
of $TS^{133}$ to the adjoint representation of $E_7/\Z_2$.
\end{enumerate}
\end{proposition}

\begin{proof}
The assertions about invariant sections are
\cite[Theorem~3.8]{CLMS24}; their parity is established in its proof.
The passage from the transitivity group to smooth invariant sections
uses the non-Abelian Liv\v{s}ic theory of
\cite[Lemma~3.7]{CekicLefeuvreHolonomy25}; see also
\cite[Proposition~3.4]{CLMS24}.
Finite Fourier content follows from \cite[Theorem~4.1]{GPSU16}, applied
to the induced orthogonal connections after complexification. In
dimension $134$, the 2nd alternative comes from the irreducible
case in \cite[Section~3.3, case~B.2.4]{CLMS24}, where the representation
is identified with the adjoint representation of $E_7/\Z_2$.
\end{proof}

All the pinching assumptions used below are preserved by finite
Riemannian covers.

\subsection{Polynomial and curvature conventions}

Let $V$ be an $n$-dimensional Euclidean space, write $S(V)$ for its
unit sphere, and set $q(y)=|y|^2$.
We identify a spherical harmonic of degree $k$ with its homogeneous
harmonic extension to $V$. The Fischer decomposition is then
$$
f=f_k+qf_{k-2}+q^2f_{k-4}+\cdots
$$
for a homogeneous polynomial $f$ of degree $k$; for tensor-valued
polynomials it is applied componentwise. All these orthogonal
projections are equivariant and hence define parallel bundle maps.
Polynomial gradients, Hessians and divergences are taken in the variable
$y$; we write $\partial_i=\partial/\partial y_i$ and
$\Delta_y=\sum_i\partial_i^2$. The symbol $\nabla$ denotes the
Levi--Civita connection and its induced connections. For a field of
fibre polynomials, horizontal differentiation is taken with the fibre
variable parallel-transported. In an orthonormal frame, with
$\nabla_i=\nabla_{e_i}$, we write
$$
\nabla_yu(x,y)=\sum_i y_i(\nabla_i u)(x,y).
$$
Its restriction to $|y|=1$ is $Xu$; we also use $X$ for this action on
homogeneous polynomial extensions. On the sphere, $\nabla_S$ denotes
the gradient of a scalar component and $\nabla^S$ the Levi--Civita
connection of the round metric on tangent tensor fields.
We normalize sphere measure $d\sigma$ to have total mass 1 and use
the nonpositive sphere Laplacian $\Delta_S$. Its sign is opposite to
the vertical Laplacian of \cite[Section~2.2]{CLMS24}. The fibrewise norm
$\|\cdot\|_{S(V)}$ is the $L^2$-norm for $d\sigma$; we omit its
subscript in purely fibrewise proofs. In differential estimates for
fields on $SM$, unsuffixed norms are $L^2$-norms for
$d\vol_g(x)d\sigma_x(v)$.

Our curvature convention is
\begin{equation}\label{eq:geo-curvature-convention}
R^g(a,b)=[\nabla_a,\nabla_b]-\nabla_{[a,b]},\qquad
\Sec_g(a,b)=\langle R^g(a,b)b,a\rangle
\end{equation}
for orthonormal $a,b$. Write
$$
A(a,b,c,d)=\langle R^g(a,b)c,d\rangle.
$$
We also regard $A$ as an operator on $\Lambda^2V$, by
$A(a,b,c,d)=\langle A(a\wedge b),c\wedge d\rangle$.
Thus $\Sec_A(a,b)=A(a,b,a,b)$ for orthonormal $a,b$ is the negative
of the corresponding sectional curvature of $g$. Equivalently, $A$
represents the negative of the geometric
curvature operator
$$
\mathcal R_g(a\wedge b,c\wedge d)
=\langle R^g(a,b)d,c\rangle.
$$
After a constant rescaling, global $\delta$-pinching gives
$\delta\leq\Sec_A\leq1$. We denote by $I$ the algebraic curvature
tensor of constant sectional curvature 1, and define the Jacobi
operator by
$$
\langle A_vz,w\rangle=A(v,z,v,w),\qquad A_vv=0.
$$
We use the contractions
$$
\langle\Ric_Az,w\rangle=\sum_i A(e_i,z,e_i,w),\qquad
\scal_A=\tr\Ric_A.
$$
The metric identifies $\Ric_A$ with its symmetric bilinear form.
For $A=-\mathcal R_g$, these are $-\Ric_g$ and $-\scal_g$,
respectively. The tensor $A$ has the usual pair symmetries and satisfies
the 1st Bianchi identity.

For a symmetric endomorphism-valued function $H$ on $S(V)$, take a
fixed orthonormal basis $(e_a)$ of $V$ and set
$z_{ab}=\nabla_SH_{ab}$. Define
\begin{equation}\label{pr:W2}
W_A^{(2)}(H)=\int_{S(V)}\sum_{a,b}
\bigl\{A(v,z_{ab},v,z_{ab})
-2A(He_b,e_a,v,z_{ab})\bigr\}\,d\sigma.
\end{equation}
For a vector-valued function $\eta$, set
$z_a=\nabla_S\langle\eta,e_a\rangle$ and define
\begin{equation}\label{pr:W1}
W_A^{(1)}(\eta)=\int_{S(V)}\sum_a
\bigl\{A(v,z_a,v,z_a)-A(\eta,e_a,v,z_a)\bigr\}\,d\sigma.
\end{equation}
For scalar functions only the Jacobi term is present. These expressions
are independent of the chosen basis. We omit the superscript when the
coefficient bundle is clear.

The covariant derivative along the geodesic flow decomposes as
$X=X_++X_-$. For $u$ of pure Fourier degree $k$, $X_\pm u$ is the degree-$(k\pm1)$
spherical harmonic component of $Xu$, with $X_-u=0$ when $k=0$.
The following version of the twisted Pestov identity will be used
throughout. Its relation with Weitzenb\"ock identities and generalized
gradients is developed in \cite{CLMSPestovWeitzenbock25}; a version on
the frame bundle was obtained in \cite{Egidi16}, and the frame bundle
and associated homogeneous fibrations are treated in
\cite{CLMSPestovFrame26}.

\begin{lemma}[Twisted Pestov identity]\label{pr:pestov}
On a closed manifold, let $u$ have pure Fourier degree $k$ and take values in the trivial
line bundle, $TM$, or $\Sym^2TM$. Then
\begin{equation}\label{eq:diff-pestov}
\int_MW_A(u)\,d\vol
=(n+2k-1)\|X_+u\|^2-(n+2k-3)\|X_-u\|^2
-\|\nabla_Hu\|^2,
\end{equation}
where $\nabla_H$ is the horizontal derivative perpendicular to $v$.
If $\nabla u$ denotes the full horizontal derivative, then
\begin{equation}\label{eq:diff-horizontal-norm}
\|\nabla u\|^2
=\|\nabla_Hu\|^2+\|X_+u\|^2+\|X_-u\|^2.
\end{equation}
\end{lemma}

\begin{proof}
Apply \cite[Proposition~3.5]{GPSU16}, equivalently
\cite[(2.13)]{CLMS24}. With our convention \eqref{eq:geo-curvature-convention},
$$
A(v,z,v,z)=-\langle R^g(z,v)v,z\rangle.
$$
The sum of the Jacobi and connection-curvature terms in that identity
is $-W_A(u)$. For $TM$ this gives \eqref{pr:W1}; on symmetric
endomorphisms the induced curvature is $H\mapsto[R^g,H]$, and the 2
output indices give the factor 2 in \eqref{pr:W2}. Rearrangement
proves \eqref{eq:diff-pestov}. Equation \eqref{eq:diff-horizontal-norm} follows by splitting the
horizontal derivative into its $v$ and $v^\perp$ components and using
the orthogonality of $X_+u$ and $X_-u$.
\end{proof}

\begin{lemma}\label{pr:admissibility}
Let $P$ be a smooth even flow-invariant orthogonal projector in $\Sym^2N$
with finite Fourier content. If its highest nonzero degree is
$k\geq4$, then $S=P_k$ satisfies
\begin{equation}\label{eq:diff-admissibility}
X_+S=0,\qquad\tr S=0.
\end{equation}
Moreover, $S(y)y$ and $y^{\mathsf T}S(y)y$, restricted to the unit
sphere, belong respectively to the harmonic spaces of degrees
$k-1$ and $k-2$.
\end{lemma}

\begin{proof}
The degree $k+1$ component of $XP=0$ is $X_+S=0$. The rank of a
smooth projector is locally constant, so its positive degree trace
components vanish. The 2 contraction statements follow from
\cite[Lemma~4.2]{CLMS24}. They also apply in dimension 8.
\end{proof}

The equation $X_+S=0$ in \eqref{eq:diff-admissibility} and the 2 contraction conditions are precisely
the conditions defining a normal twisted conformal Killing tensor in
\cite[Definition~4.3]{CLMS24}.
In particular, this terminology does not require $S(y)y=0$. We will
apply the estimates below to components for which this stronger
identity holds.

\subsection{Differential estimates}

We use $(a_1,\ldots,a_s)_0$ for the trace-free orthogonal tensor
representation with Young diagram of row lengths $a_1,\ldots,a_s$.
A harmonic symmetric matrix
polynomial $H$ of degree $k$ with $\tr H=0$ and $H(y)y=0$ realizes
$(k,2)_0$: the latter equation is the Young relation, and
$\Delta_y(Hy)=2\operatorname{div}_yH=0$ removes the mixed trace. A normal
harmonic vector polynomial $\eta$ of degree $\ell$, meaning
$y\cdot\eta(y)=0$, similarly realizes $(\ell,1)_0$.
For $n\geq8$, these representations are irreducible,
since their Young diagrams have at most 3 rows, fewer than $n/2$.
We use the vector tensor product rule recalled in
\cite[Section~2]{CGH00}. The vector estimate in
Lemma~\ref{pr:vector-twistor} is proved by explicit orthogonal projection
and holds for every $n\geq4$.

\begin{lemma}\label{pr:tensor-twistor}
Let $n\geq8$ and $k\geq2$. Let $H$ be a smooth field of harmonic
symmetric matrix polynomials of degree $k$, with $\tr H=0$ and
$H(y)y=0$. Suppose that the $(k+1,2)_0$ component of $\nabla H$
vanishes. Then
\begin{equation}\label{eq:diff-tensor-ratio}
\|X_+H\|^2\leq\frac{k-1}{n+k-1}\|X_-H\|^2
\end{equation}
and
\begin{equation}\label{eq:diff-tensor-bound}
\int_MW_A^{(2)}(H)\,d\vol
\leq-\|\nabla H\|^2
-\frac{(n-2)(n+2k-2)}{n+k-1}\|X_-H\|^2\leq0.
\end{equation}
\end{lemma}

\begin{proof}
Write
$$
\nabla_yH(y)=U(y)+qV(y),
$$
where $U,V$ are harmonic and trace-free, of degrees $k+1,k-1$.
Their restrictions to the sphere are $X_+H,X_-H$. Differentiating
$Hy=0$ horizontally gives $Uy=-qVy$. Put $Q=Vy$, of degree $k$.
Since $U,V$ are harmonic, $\Delta_y^2(Uy)=\Delta_y^2Q=0$.
Applying $\Delta_y^2$ to $Uy=-qQ$ therefore gives
$4(n+2k-2)\Delta_y Q=0$. Hence $Q$ is harmonic,
$\operatorname{div}_yV=0$, and 1 application of the Laplacian gives
$$
\operatorname{div}_yU=-(n+2k)Q.
$$
The scalar $h=y\cdot Q=y^{\mathsf T}Vy$ is harmonic because
$\tr V=\operatorname{div}_yV=0$. Applying $\Delta_y$ to
$y^{\mathsf T}Uy=-qh$ yields
$$
-4(n+2k)h=-2(n+2k+2)h,
$$
so $h=0$ and $Q$ is normal.

Set $b=n+2k-2$ and $(DQ)_{ij}=\partial_iQ_j+\partial_jQ_i$.
The harmonic trace-free polynomial
$$
U_Q=yQ^{\mathsf T}+Qy^{\mathsf T}-\frac q b DQ
$$
satisfies
$$
U_Qy=\frac{n+k-1}{b}qQ,\qquad (DQ)y=(k-1)Q.
$$
It follows that
\begin{equation}\label{eq:diff-tensor-decomposition}
U=-\frac b{n+k-1}U_Q+U_0,
\qquad V=\frac1{k-1}DQ+V_0,
\end{equation}
where $U_0,V_0$ are normal harmonic symmetric polynomials. The 1st
is the image of the $(k+1,2)_0$ component of $\nabla H$ under the
harmonic symbol map and is 0 by assumption.
The 2nd is orthogonal to $DQ$, by the Young and mixed trace
relations, or by integration by parts on the sphere.

Normality and Euler's identity give, on the unit sphere,
$$
|yQ^{\mathsf T}+Qy^{\mathsf T}|^2=2|Q|^2,
\qquad
\langle yQ^{\mathsf T}+Qy^{\mathsf T},DQ\rangle
=2(k-1)|Q|^2.
$$
Since $U_Q$ and $DQ$ have different harmonic degrees, they are
orthogonal. Consequently,
$$
\|DQ\|^2=2b(k-1)\|Q\|^2,\qquad
\|U_Q\|^2=\frac{2(n+k-1)}b\|Q\|^2.
$$
The decompositions \eqref{eq:diff-tensor-decomposition} now give \eqref{eq:diff-tensor-ratio}. Finally,
\eqref{eq:diff-pestov} and \eqref{eq:diff-horizontal-norm} give
$$
\int_MW_A^{(2)}(H)\,d\vol
=-\|\nabla H\|^2+(n+2k)\|X_+H\|^2
-(n+2k-4)\|X_-H\|^2.
$$
Substitution of \eqref{eq:diff-tensor-ratio} proves \eqref{eq:diff-tensor-bound}.
\end{proof}

\begin{lemma}\label{pr:vector-twistor}
Let $n\geq4$ and $\ell\geq1$. Let $\eta$ be a smooth field of normal
harmonic vector polynomials of degree $\ell$. Suppose that the normal
harmonic degree-$\ell+1$ projection of $\nabla_y\eta(y)$ vanishes.
Then
\begin{equation}\label{eq:diff-vector-ratio}
\|X_+\eta\|^2\leq\frac\ell{n+\ell-2}\|X_-\eta\|^2
\end{equation}
and
\begin{equation}\label{eq:diff-vector-bound}
\int_MW_A^{(1)}(\eta)\,d\vol
\leq-\|\nabla\eta\|^2
-\frac{(n-4)(n+2\ell-2)}{n+\ell-2}\|X_-\eta\|^2\leq0.
\end{equation}
\end{lemma}

\begin{proof}
Write $\nabla_y\eta=U+qV$, with $U,V$ harmonic of degrees
$\ell+1,\ell-1$. As in the preceding proof, the identity
$y\cdot U=-q(y\cdot V)$ shows that $T=y\cdot V$ is harmonic of
degree $\ell$. Set $b=n+2\ell-2$. The harmonic polynomial
$$
U_T=yT-\frac q b\operatorname{grad}_yT
$$
has contraction $y\cdot U_T=(n+\ell-2)qT/b$. It is orthogonal to
normal harmonic vectors of degree $\ell+1$: the term $yT$ pairs to
0 by normality, and $\operatorname{grad}_yT$ has degree $\ell-1$.
Thus
$$
U=-\frac b{n+\ell-2}U_T,
\qquad V=\frac1\ell\operatorname{grad}_yT+V_0,
$$
where $V_0$ is normal and harmonic. It is orthogonal to
$\operatorname{grad}_yT$, since a normal harmonic vector is
divergence-free on the sphere. Euler's identity and harmonic
orthogonality give
$$
\|\operatorname{grad}_yT\|^2=\ell b\|T\|^2,
\qquad \|U_T\|^2=\frac{n+\ell-2}{b}\|T\|^2.
$$
This proves \eqref{eq:diff-vector-ratio}. Substituting it in
$$
\int_MW_A^{(1)}(\eta)\,d\vol
=-\|\nabla\eta\|^2+(n+2\ell)\|X_+\eta\|^2
-(n+2\ell-4)\|X_-\eta\|^2
$$
proves \eqref{eq:diff-vector-bound}.
\end{proof}

For scalar coefficients, the following is the classical vanishing
theorem for trace-free conformal Killing tensors in negative curvature;
see \cite[Theorem~1.6]{DairbekovSharafutdinov10} and
\cite[Proposition~6.6]{HeilMoroianuSemmelmann16}. We include its short
Pestov proof.

\begin{lemma}\label{pr:scalar-vanishing}
On a closed negatively curved manifold, a scalar harmonic section
$\tau$ of positive degree satisfying $X_+\tau=0$ is 0.
\end{lemma}

\begin{proof}
After a constant rescaling, choose $0<\delta\leq1$ such that
$\delta\leq\Sec_A\leq1$. If $\tau$ has degree $j>0$, the Jacobi term satisfies
\begin{equation}\label{eq:diff-scalar-lower-bound}
W_A(\tau)\geq\delta j(n+j-2)\|\tau\|_{S(V)}^2.
\end{equation}
Its integral is nonpositive by \eqref{eq:diff-pestov}. The lower bound \eqref{eq:diff-scalar-lower-bound}
therefore implies $\tau=0$.
\end{proof}

\begin{proposition}\label{pr:vector-coercivity}
Let $n\geq4$, $\ell\geq3$, and let $\eta$ be a normal harmonic vector
polynomial of degree $\ell$. If $\delta\leq\Sec_A\leq1$ and
$\delta\geq1/4$, then, with $L=\ell(n+\ell-2)-1$,
\begin{equation}\label{eq:diff-vector-coercivity}
W_A^{(1)}(\eta)
\geq\frac{13L}{32(3n+2)}\|\eta\|_{S(V)}^2.
\end{equation}
On a closed strictly quarter-pinched manifold, a smooth field satisfying
the hypothesis of Lemma~\ref{pr:vector-twistor} is consequently 0.
\end{proposition}

\begin{proof}
At $v\in S(V)$ take an orthonormal frame $e_0=v,e_1,\ldots,e_{n-1}$
and write
$$
Z_{ab}=\langle\nabla^S_{e_b}\eta,e_a\rangle,
\qquad 1\leq a,b\leq n-1.
$$
Differentiating $\langle\eta,v\rangle=0$ gives $z_0=-\eta$.
The associated Jacobi and twisting terms in \eqref{pr:W1} cancel.
Hence $W_I^{(1)}(\eta)=\int|Z|^2=L\|\eta\|^2$.
For $\delta<1$ write $A=\delta I+(1-\delta)E$, with
$0\leq\Sec_E\leq1$. The remaining $E$-Jacobi term is nonnegative;
the twisting term is the pairing of $Z$ with
$$
M_{ab}=E(\eta,e_a,v,e_b).
$$

At a point where $\eta\ne0$, normalize $|\eta|=1$ and choose
$e_1=\eta$. Relative to $\R\eta\oplus\{v,\eta\}^\perp$, the matrix
has the form
$$
M=\begin{pmatrix}0&0\\ k&B\end{pmatrix}.
$$
For a unit $x\perp v,\eta$, the value
$\langle B_{\mathrm{sym}}x,x\rangle=E(\eta,x,v,x)$ is an
off-diagonal entry of a Jacobi operator with spectrum in $[0,1]$.
It has absolute value at most $1/2$, so
$|B_{\mathrm{sym}}|^2\leq(n-2)/4$.
The 1st Bianchi identity gives
$$
B_{ab}-B_{ba}=E(\eta,v,e_a,e_b).
$$
Berger's curvature estimate
\cite[Lemma~3.7]{BourguignonKarcher78}, in the normalization of
\cite[Lemma~2.1]{CLMS24}, bounds the right-hand side by $2/3$ on
orthonormal pairs. Thus
every singular value of $B_{\mathrm{skew}}$ is at most $1/3$, and
$|B_{\mathrm{skew}}|^2\leq(n-2)/9$.
Finally, $k$ is, up to sign, the off-diagonal component of the Jacobi
operator in direction $\eta$ applied to $v$. If $0\leq J\leq\Id$
and $a=\langle Jv,v\rangle$, then
$$
|k|^2\leq\langle J^2v,v\rangle-a^2\leq a-a^2\leq\frac14.
$$
Each matrix projection contains $|k|^2/2$. Restoring $|\eta|$ gives
\begin{equation}\label{eq:diff-vector-twisting-bounds}
|M_{\mathrm{sym}}|^2\leq\frac{2n-3}{8}|\eta|^2,
\qquad |M_{\mathrm{skew}}|^2\leq\frac{8n-7}{72}|\eta|^2.
\end{equation}

Harmonicity and normality imply that $\eta$ is divergence-free on the
sphere. Integration by parts on $S^{n-1}$ therefore gives
$$
\int_{S(V)}\sum_{a,b}Z_{ab}Z_{ba}\,d\sigma
=-(n-2)\|\eta\|^2.
$$
Consequently,
\begin{equation}\label{eq:diff-vector-derivative-norms}
\|Z_{\mathrm{sym}}\|^2=\frac{L-n+2}{2}\|\eta\|^2,
\qquad
\|Z_{\mathrm{skew}}\|^2=\frac{L+n-2}{2}\|\eta\|^2.
\end{equation}
Cauchy--Schwarz with \eqref{eq:diff-vector-twisting-bounds} and \eqref{eq:diff-vector-derivative-norms} now yields
\begin{equation}\label{eq:diff-vector-curvature-bound}
W_A^{(1)}(\eta)\geq\bigl[\delta L-(1-\delta)B(L)\bigr]\|\eta\|^2,
\end{equation}
where
\begin{equation}\label{eq:diff-vector-error}
B(L)=\sqrt{\frac{(2n-3)(L-n+2)}{16}}
+\sqrt{\frac{(8n-7)(L+n-2)}{144}}.
\end{equation}
For $\ell=3$ we have $L_3=3n+2$ and \eqref{eq:diff-vector-error} gives
$$
B(L_3)
=\sqrt{\frac{(2n-3)(n+2)}8}
+\sqrt{\frac{(8n-7)n}{36}}
<n+\frac18.
$$
Indeed, the difference between the square of $n/2+1/8$ and the
1st radicand is $49/64$, and the 2nd square root is less than
$n/2$. Hence $L_3/4-3B(L_3)/4\geq13/32$.
Both $\sqrt{L-n+2}/L$ and $\sqrt{L+n-2}/L$ decrease for
$L>2(n-2)$. Since $L\geq L_3$, it follows that
\begin{equation}\label{eq:diff-vector-error-ratio}
\frac{B(L)}L\leq\frac{B(L_3)}{L_3}.
\end{equation}
Using \eqref{eq:diff-vector-error-ratio} in \eqref{eq:diff-vector-curvature-bound} proves \eqref{eq:diff-vector-coercivity} at $\delta=1/4$, and the coefficient
increases with $\delta$. The case $\delta=1$ follows directly from
$W_I^{(1)}(\eta)=L\|\eta\|^2$. The final assertion follows from
\eqref{eq:diff-vector-bound} and \eqref{eq:diff-vector-coercivity}.
\end{proof}

\subsection{Curvature estimates for symmetric tensors}

\begin{lemma}\label{pr:tensor-high-curvature}
Let $k\geq2$ and let $H$ be a harmonic symmetric matrix polynomial
of degree $k$, with $\tr H=0$ and $H(y)y=0$. Put
$L=k(n+k-2)-2$. If $\delta\leq\Sec_A\leq1$, then
\begin{equation}\label{eq:curv-high-degree-bound}
W_A^{(2)}(H)\geq
\left[\delta L-\frac43(1-\delta)
\sqrt{\frac{(n-1)L}{2}}\right]\|H\|_{S(V)}^2.
\end{equation}
The coefficient is positive if $n\geq8$, $k\geq8$ and
$\delta\geq1/4$.
\end{lemma}

\begin{proof}
At $v=e_0$, normality gives $H_{0j}=0$ and
$z_{0j}=z_{j0}=-He_j$. In \eqref{pr:W2}, the contribution indexed by
$(0,j)$ is $-A(v,He_j,v,He_j)$, and that indexed by $(j,0)$ is its
opposite. The radial terms therefore cancel. If $Z=\nabla^SH$ is
the covariant derivative of the tangent symmetric tensor on the
sphere, then
\begin{equation}\label{eq:curv-tangent-energy}
\|Z\|^2=\bigl[k(n+k-2)-2\bigr]\|H\|^2=L\|H\|^2.
\end{equation}
Here the subtracted term is the radial derivative energy
$2\|H\|^2$.

For $\delta<1$, write $A=\delta I+(1-\delta)E$,
$0\leq\Sec_E\leq1$, and discard the nonnegative tangential
$E$-Jacobi term. Diagonalizing $H|_{v^\perp}$ with eigenvalues
$h_a$, the absolute value of the remaining twisting term is bounded by
$$
2\sum_{a<b}|h_a-h_b|\,
|E(e_a,e_b,v,z_{ab})|.
$$
The constant-curvature tensor has 0 value on these arguments,
so the same term is obtained from $E-I/2$. The curvature estimate
\cite[Lemma~3.7]{BourguignonKarcher78}, or
\cite[Lemma~2.1]{CLMS24}, gives
$$
|E(e_a,e_b,v,z_{ab})|\leq\frac23|z_{ab}|.
$$
Since $H$ is trace-free,
$\sum_{a<b}(h_a-h_b)^2=(n-1)|H|^2$.
The off-diagonal entries occur 2 times in $|Z|^2$, and the claimed
estimate \eqref{eq:curv-high-degree-bound} follows from Cauchy--Schwarz. If $\delta=1$, it follows
from $W_I^{(2)}(H)=L\|H\|^2$.
At $\delta=1/4$, positivity is equivalent to $L>8(n-1)$.
For $k=8$ the difference is $54$, and it increases with $k$.
The coefficient also increases with $\delta$.
\end{proof}

Degree 6 requires a sharper estimate of the connection-curvature
term. We first prove 3 algebraic lemmas. Indices in their tensor
norms are summed over ordered tuples.

\begin{lemma}\label{pr:magnetic-identities}
Let $C:W\times W\to W$ be a symmetric bilinear map on a
$d$-dimensional Euclidean space. In an orthonormal basis, write
$C_{a;bc}=\langle C(e_b,e_c),e_a\rangle$, and suppose that
$$
C_{a;bc}+C_{b;ca}+C_{c;ab}=0,
\qquad |C(x,x)|\leq\frac12|x|^2.
$$
Let $H$ be trace-free and symmetric, with eigenvalues $h_a$ in an
orthonormal basis. Set
$$
B_{ab,c}=\frac23(C_{a;bc}-C_{b;ac}),
\qquad M_{ab;c}=(h_a-h_b)B_{ab,c}.
$$
Then $M$ is symmetric in $a,b$. Denote its completely symmetric
projection by $M_S$ and its Young-symmetry $(2,1)$ projection by
$M_T=M-M_S$, and put
$$
\mathsf S=|M_S|^2,\qquad \mathsf T=|M_T|^2,
\qquad
\mathsf D=\sum_{a,b,c}(h_a+h_b+h_c)^2C_{a;bc}^2.
$$
Then
\begin{equation}\label{pr:joint-budget}
3\mathsf S+2\mathsf T+\frac49\mathsf D
\leq2(d+2)|H|^2.
\end{equation}
Moreover,
\begin{equation}\label{pr:hook-graph}
\mathsf T=\frac49\sum_{a<b}(h_a-h_b)^2L_{ab},
\qquad
L_{ab}=\sum_c(C_{a;bc}^2+C_{b;ac}^2+C_{c;ab}^2).
\end{equation}
\end{lemma}

\begin{proof}
The cyclic relation gives
\begin{equation}\label{eq:curv-symmetric-projection}
(M_S)_{ab;c}=\frac23
\bigl(h_aC_{a;bc}+h_bC_{b;ac}+h_cC_{c;ab}\bigr).
\end{equation}
Squaring \eqref{eq:curv-symmetric-projection} and relabeling indices yields
\begin{equation}\label{eq:curv-projection-norms}
\begin{aligned}
\mathsf S&=\frac43\sum_{a,b,c}
(h_a^2-2h_ah_b+h_bh_c)C_{a;bc}^2,\\
\mathsf T&=\frac49\sum_{a,b,c}
(h_a^2+2h_b^2-2h_ah_b-h_bh_c)C_{a;bc}^2,\\
\mathsf D&=\sum_{a,b,c}
(h_a^2+2h_b^2+4h_ah_b+2h_bh_c)C_{a;bc}^2.
\end{aligned}
\end{equation}
For example, the mixed contraction used in the 1st equality of \eqref{eq:curv-projection-norms} is
$$
\sum_{a,b,c}h_ah_cC_{a;bc}C_{c;ba}
=\sum_{a,b,c}\left(-h_ah_b+\frac12h_bh_c\right)C_{a;bc}^2,
$$
which follows by inserting
$C_{c;ba}=-C_{a;bc}-C_{b;ac}$ and using the cyclic relation once more.

For $z\in W$, let $C_z$ be the symmetric matrix
$(C_z)_{bc}=\sum_a z_aC_{a;bc}$. The hypothesis implies
$\|C_z\|_{\mathrm{op}}\leq|z|/2$. If $G$ is a standard Gaussian
vector and $(u_H)_a=\sum_bh_bC_{a;bb}$, the Gaussian 4th-moment
formula gives
\begin{equation}\label{eq:curv-gaussian-moments}
\begin{aligned}
\mathbb E|C_{HG}G|^2
&=\sum_{a,b,c}\left(h_a^2-h_ah_b+\frac12h_bh_c\right)C_{a;bc}^2
+\frac14|u_H|^2,\\
\mathbb E|C_G(HG)|^2
&=\sum_{a,b,c}\left(h_b^2-h_ah_b+\frac12h_bh_c\right)C_{a;bc}^2
+\frac14|u_H|^2.
\end{aligned}
\end{equation}
These identities \eqref{eq:curv-gaussian-moments} follow by symmetrizing the 2 input indices of
each quadratic vector and expanding the Gaussian 4th moments.
In both cases the resulting bilinear map has trace $-u_H/2$, by the
cyclic relation.
Combining \eqref{eq:curv-projection-norms} and \eqref{eq:curv-gaussian-moments} gives
\begin{equation}\label{eq:curv-budget-identity}
3\mathsf S+2\mathsf T+\frac49\mathsf D
=\frac83\left(2\mathbb E|C_{HG}G|^2
+\mathbb E|C_G(HG)|^2-\frac34|u_H|^2\right).
\end{equation}
Each expectation is at most
$$
\frac14\mathbb E\bigl(|HG|^2|G|^2\bigr)
=\frac{d+2}{4}|H|^2.
$$
Dropping the final nonpositive term in \eqref{eq:curv-budget-identity} proves \eqref{pr:joint-budget}.

The same expansion rewrites the 2nd projection norm as
$$
\mathsf T=\frac29\sum_{a,b,c}
\bigl[(h_a-h_b)^2+(h_b-h_c)^2+(h_c-h_a)^2\bigr]C_{a;bc}^2.
$$
Collecting each unordered pair, including repeated index terms,
proves \eqref{pr:hook-graph}.
\end{proof}

\begin{lemma}\label{pr:magnetic-edge}
In the notation of Lemma~\ref{pr:magnetic-identities}, suppose $d\geq3$.
Fix orthonormal $e_a,e_b$, extend them to an orthonormal basis of $W$,
and put $U=\{e_a,e_b\}^\perp$. Set
$$
P=C_{a;bb},\quad Q=C_{b;aa},\quad r_{ab}=P^2+Q^2,
$$
$$
u_c=C_{a;bc},\quad v_c=C_{b;ac},\quad w_c=C_{c;ab}\quad(e_c\in U),
\qquad D_{ab}=|u|^2+|v|^2+|w|^2.
$$
Then
$$
D_{ab}+\frac35r_{ab}\leq\frac23,
$$
and
\begin{equation}\label{eq:curv-edge-bounds}
L_{ab}\leq
\min\left\{\frac23+\frac9{10}r_{ab},
\frac34-\frac14r_{ab}\right\}\leq\frac{101}{138}.
\end{equation}
\end{lemma}

\begin{proof}
Let $(\Z/2)^3$ act by independent sign changes on
$\R e_a\oplus\R e_b\oplus U$, and use the induced action on $C$.
Write $C^a,C^b,C^{abU}$ for the character projections odd only on
the 1st summand, odd only on the 2nd, and odd on all 3.
For fixed $x$, vector-valued Parseval on this finite group gives
\begin{equation}\label{eq:curv-character-parseval}
\sum_\chi|C^\chi(x,x)|^2
=\frac18\sum_g|(g\cdot C)(x,x)|^2\leq\frac14|x|^4.
\end{equation}
In particular, the sum of the 3 selected terms has the same bound.
This identity uses orthogonality of the group characters.

Write $x=se_a+te_b+z$, $z\in U$. The cyclic relation gives
$u+v+w=0$ and
$$
C^{abU}(x,x)
=2t\langle u,z\rangle e_a+2s\langle v,z\rangle e_b+2st w.
$$
The character projection $C^a$ has the form
$$
C^a(x,x)=(Pt^2+\langle Rz,z\rangle)e_a-Pst e_b+2sAz,
$$
where $R,A\in\End(U)$, $R$ is symmetric, and
$\langle Az,z\rangle=-\langle Rz,z\rangle/2$.
The formula for $C^b$ is obtained by interchanging $a,b$ and $P,Q$.
These forms follow directly from the character parities; the
$aaa$ term vanishes by the cyclic relation.

Choose $s^2=t^2=3/8$ and $|z|^2=1/4$, so $|x|=1$.
Writing $x_0=\langle Rz,z\rangle$, Cauchy--Schwarz gives
$|2sAz|^2\geq3x_0^2/2$. Therefore
$$
|C^a(x,x)|^2\geq\frac9{64}P^2
+\min_{x_0}\left[\left(\frac{3P}{8}+x_0\right)^2
+\frac32x_0^2\right]=\frac9{40}P^2.
$$
Similarly, $|C^b(x,x)|^2\geq9Q^2/40$.
We may choose the direction of $z$ to maximize the quadratic form
$uu^{\mathsf T}+vv^{\mathsf T}$. Since this matrix is at least
$(u-v)(u-v)^{\mathsf T}/2$, we obtain
$$
|C^{abU}(x,x)|^2\geq\frac3{16}|u-v|^2
+\frac9{16}|u+v|^2=\frac38D_{ab}.
$$
Parseval \eqref{eq:curv-character-parseval} now gives $9r_{ab}/40+3D_{ab}/8\leq1/4$.

The repeated-index relations are
$C_{b;ab}=-P/2$ and $C_{a;ab}=-Q/2$. Thus
$$
L_{ab}=D_{ab}+\frac32r_{ab}
\leq\frac23+\frac9{10}r_{ab}.
$$
For a 2nd bound, the 1st 2 sums in $L_{ab}$ are row norms of
$C_{e_a}$ and $C_{e_b}$, each at most $1/4$. Applying the hypothesis
to $(e_a+e_b)/\sqrt2$ and $(e_a-e_b)/\sqrt2$, then averaging, gives
$$
|C(e_a,e_b)|^2
+\frac14|C(e_a,e_a)+C(e_b,e_b)|^2\leq\frac14.
$$
The $e_a,e_b$ components of the vector in the 2nd term have
squared norm $r_{ab}$. Hence
$L_{ab}\leq3/4-r_{ab}/4$.
The 2 affine bounds intersect at $r_{ab}=5/69$, where their common
value is $101/138$; their minimum cannot exceed this value.
\end{proof}

\begin{lemma}\label{pr:magnetic-joint}
Under the hypotheses of Lemma~\ref{pr:magnetic-identities}, with $d\geq3$,
$$
\mathsf T\leq\frac{202d}{621}|H|^2,
\qquad
3\mathsf S+\frac{28}{9}\mathsf T
\leq\left(\frac{566d}{243}+7\right)|H|^2.
$$
\end{lemma}

\begin{proof}
Trace-freeness gives
$\sum_{a<b}(h_a-h_b)^2=d|H|^2$.
The 1st inequality follows from \eqref{pr:hook-graph} and
Lemma~\ref{pr:magnetic-edge}, via \eqref{eq:curv-edge-bounds}.
For the 2nd, retain the terms with repeated indices in the
$(2,1)$-component norm:
$$
\mathsf R=\frac23\sum_{i\ne j}(h_i-h_j)^2C_{i;jj}^2.
$$
The 1st affine estimate for $L_{ab}$ in \eqref{eq:curv-edge-bounds} gives
\begin{equation}\label{eq:curv-repeated-index-estimate}
\mathsf T\leq\frac{8d}{27}|H|^2+\frac35\mathsf R.
\end{equation}
The repeated-index part of $\mathsf D$ is
$$
\frac32\sum_{i\ne j}(h_i+2h_j)^2C_{i;jj}^2\leq\mathsf D.
$$
Using
$$
(h_i-h_j)^2\leq\frac32(h_i+2h_j)^2+27h_j^2,
\qquad \sum_iC_{i;jj}^2\leq\frac14,
$$
we obtain $\mathsf R\leq2\mathsf D/3+9|H|^2/2$.
Substitution in \eqref{pr:joint-budget} yields
$$
3\mathsf S+2\mathsf T+\frac23\mathsf R
\leq(2d+7)|H|^2.
$$
Finally,
$\mathsf R\geq5\mathsf T/3-40d|H|^2/81$,
by the earlier estimate \eqref{eq:curv-repeated-index-estimate} for $\mathsf T$. Combining these inequalities
proves the result.
\end{proof}

\begin{theorem}\label{pr:sextic}
Let $n\geq8$ and let $H$ be a harmonic symmetric matrix polynomial
of degree 6, with $\tr H=0$ and $H(y)y=0$. If
$\delta\leq\Sec_A\leq1$ and $\delta\geq1/4$, then
\begin{equation}\label{eq:curv-sextic-bound}
W_A^{(2)}(H)\geq\frac{127}{80}\|H\|_{S(V)}^2.
\end{equation}
\end{theorem}

\begin{proof}
The case $\delta=1$ follows from $W_I^{(2)}(H)=(6n+22)\|H\|^2$.
Assume $\delta<1$ and write $A=\delta I+(1-\delta)E$.
At $v\in S(V)$ put $W=v^\perp$, $d=n-1$, and set
$$
B_{ab,c}=E(e_a,e_b,v,e_c),\qquad
C_{a;bc}=\frac12(B_{ab,c}+B_{ac,b}).
$$
The 1st Bianchi identity gives the cyclic relation and inversion
formula in Lemma~\ref{pr:magnetic-identities}. The vector $C(x,x)$ is the
$W$ component of the Jacobi operator in direction $x$ applied to $v$.
The Jacobi operator has spectrum in $[0,|x|^2]$, so the variance estimate gives
$|C(x,x)|\leq|x|^2/2$. The preceding 3 lemmas therefore apply.

Regard $H$ as a tangent symmetric tensor on $S^{n-1}$ and write
$Z_{ab;c}=(\nabla_c^SH)_{ab}$. Harmonicity and normality imply that
$H$ is divergence-free on the sphere. Together with $\tr H=0$, this
shows that all metric traces of $Z$ vanish. By the radial cancellation
in Lemma~\ref{pr:tensor-high-curvature} (see \eqref{eq:curv-tangent-energy}),
$$
\|Z\|^2=(6d+28)\|H\|^2.
$$
Integration by parts and the sphere curvature identity give
$$
\int_{S(V)}\sum_{a,b,c}Z_{ab;c}Z_{cb;a}\,d\sigma=-d\|H\|^2.
$$
Indeed, divergence 0 changes the left-hand side to
$-\int_{S(V)}\sum_{a,b,c}H_{ab}\allowbreak
[\nabla_c^S,\nabla_a^S]H_{cb}\,d\sigma$.
The 2 covariant indices contribute $(d-1)H_{ab}$ and
$H_{ab}-\delta_{ab}\tr H$, respectively, giving $dH_{ab}$.
Write $Z_S$ for the complete symmetrization of $Z$ and $Z_T=Z-Z_S$.
Complete symmetrization averages the 3 indices, hence
\begin{equation}\label{eq:curv-sextic-projection-energy}
\|Z_S\|^2=\frac{4d+28}{3}\|H\|^2,
\qquad \|Z_T\|^2=\frac{14d+56}{3}\|H\|^2.
\end{equation}

In an eigenbasis of $H$, the twisting term is $-\langle M,Z\rangle$,
with $M$ as in Lemma~\ref{pr:magnetic-identities}. Normalize
$\|H\|^2=1$ and denote the integrated squared norms of $M_S,M_T$
by $S,T$. Cauchy--Schwarz, \eqref{eq:curv-sextic-projection-energy}, and Lemma~\ref{pr:magnetic-joint} bound the
absolute value of this pairing by
\begin{equation}\label{eq:curv-sextic-objective}
\sqrt{S\frac{4d+28}{3}}+\sqrt{T\frac{14d+56}{3}},
\end{equation}
where
\begin{equation}\label{eq:curv-sextic-feasible-bounds}
0\leq T\leq\frac{202d}{621},\qquad
0\leq S\leq\frac{566d}{729}+\frac73-\frac{28}{27}T.
\end{equation}
The objective \eqref{eq:curv-sextic-objective} increases with $S$. On the affine upper bound for $S$ in \eqref{eq:curv-sextic-feasible-bounds}
it is concave in $T$, and its derivative at $T=202d/621$ is positive.
After squaring positive quantities, the latter assertion is equivalent
to
$$
14(d+4)S-4\left(\frac{28}{27}\right)^2(d+7)T>0.
$$
At that endpoint, $S=818d/1863+7/3$, and the left-hand side equals
$$
\frac{14(153526d^2+1534681d+4225284)}{452709}>0.
$$
The maximum is therefore at the endpoint and is bounded by
\begin{equation}\label{eq:curv-sextic-twisting-bound}
B_d=
\sqrt{\left(\frac{818d}{1863}+\frac73\right)\frac{4d+28}{3}}
+\sqrt{\frac{202d}{621}\frac{14d+56}{3}}.
\end{equation}
For $d\geq0$,
\begin{equation}\label{eq:curv-sextic-root-bounds}
\sqrt{\left(\frac{818d}{1863}+\frac73\right)\frac{4d+28}{3}}
\leq\frac{23d}{30}+\frac{19}{4},
\qquad
\sqrt{\frac{202d}{621}\frac{14d+56}{3}}
\leq\frac{37(d+2)}{30}.
\end{equation}
For the 1st inequality in \eqref{eq:curv-sextic-root-bounds}, subtracting the squared left-hand side
from the squared right-hand side gives
$$
\frac{1309d^2}{558900}+\frac{8291d}{111780}+\frac{113}{144}>0.
$$
The 2nd follows from $\sqrt{d(d+4)}\leq d+2$ and
$ (37/30)^2-2828/1863=583/186300>0$.
Thus \eqref{eq:curv-sextic-twisting-bound} and \eqref{eq:curv-sextic-root-bounds} give $B_d\leq2d+433/60$.

Discarding only the nonnegative tangential $E$-Jacobi term gives
\begin{equation}\label{eq:curv-sextic-curvature-bound}
W_A^{(2)}(H)\geq
\bigl[\delta(6d+28)-(1-\delta)B_d\bigr]\|H\|^2.
\end{equation}
The coefficient in \eqref{eq:curv-sextic-curvature-bound} increases with $\delta$. At $\delta=1/4$ it is at
least
$$
\frac{6d+28}{4}-\frac34\left(2d+\frac{433}{60}\right)
=\frac{127}{80},
$$
as required by \eqref{eq:curv-sextic-bound}.
\end{proof}

\subsection{Vanishing of the highest Fourier component}

\begin{lemma}\label{pr:three-components}
Let $n\geq8$, $k\geq4$, and let $S$ be a harmonic homogeneous
symmetric matrix polynomial of degree $k$ satisfying the trace and
contraction conditions of Lemma~\ref{pr:admissibility}. Then there is a
unique orthogonal decomposition
\begin{equation}\label{eq:curv-three-component-splitting}
S=H+\mathcal J(\eta)+\mathcal I(\tau),
\end{equation}
where $H$ is harmonic, trace-free and normal of degree $k$,
$\eta$ is a normal harmonic vector of degree $k-1$, and $\tau$ is
a harmonic scalar of degree $k-2$. The 3 summands have types
$$
(k,2)_0,\qquad(k-1,1)_0,\qquad(k-2)_0,
$$
respectively, and the embeddings are parallel.
\end{lemma}

\begin{proof}
Put $d=k-2$ and $a=n+2d$. The embeddings are
\begin{equation}\label{eq:curv-vector-embedding}
\mathcal J(\eta)
=y\eta^{\mathsf T}+\eta y^{\mathsf T}-\frac q aD\eta,
\end{equation}
\begin{equation}\label{eq:curv-scalar-embedding}
\mathcal I(\tau)
=yy^{\mathsf T}\tau
-\frac q a\bigl(y(\operatorname{grad}_y\tau)^{\mathsf T}
+(\operatorname{grad}_y\tau)y^{\mathsf T}+\Id\,\tau\bigr)
+\frac{q^2}{a(a-2)}\operatorname{Hess}\tau,
\end{equation}
where $(D\eta)_{ij}=\partial_i\eta_j+\partial_j\eta_i$.
Normality and harmonicity imply $\operatorname{div}_y\eta=0$.
The identities
$$
\Delta_y(y\eta^{\mathsf T}+\eta y^{\mathsf T})=2D\eta,
\qquad \Delta_y(qD\eta)=2aD\eta
$$
show that $\mathcal J(\eta)$ is harmonic and trace-free.
Euler's identity gives
\begin{equation}\label{eq:curv-vector-contraction}
\mathcal J(\eta)y=\frac{n+d}{a}q\eta.
\end{equation}
To check the scalar embedding, abbreviate the parenthesized expression
by $U$. Product differentiation gives
$$
\begin{aligned}
\Delta_y(yy^{\mathsf T}\tau)&=2U,\\
\Delta_y(qU)&=2aU+4q\operatorname{Hess}\tau,\\
\Delta_y(q^2\operatorname{Hess}\tau)&=4(a-2)q\operatorname{Hess}\tau.
\end{aligned}
$$
These terms cancel in $\Delta_y\mathcal I(\tau)$.
Its trace is 0 because $\tr U=a\tau$ and
$\tr\operatorname{Hess}\tau=0$.
Euler's identities also give
\begin{equation}\label{eq:curv-scalar-contraction}
\mathcal I(\tau)y
=\frac{n+d-1}{a}q
\left(y\tau-\frac q{a-2}\operatorname{grad}_y\tau\right),
\end{equation}
\begin{equation}\label{eq:curv-scalar-double-contraction}
y^{\mathsf T}\mathcal I(\tau)y
=\frac{(n+d-1)(n+d-2)}{a(a-2)}q^2\tau.
\end{equation}
The vector in parentheses is harmonic.

By the contraction assumptions, $Sy=qQ$ and $y\cdot Q=qt$,
where $Q,t$ are harmonic of degrees $k-1,k-2$.
Subtract the unique scalar summand that removes the 2-fold
contraction, and then the unique vector summand that removes the
remaining contraction. The remainder is $H$. The nonzero
coefficients in the contractions \eqref{eq:curv-vector-contraction}, \eqref{eq:curv-scalar-contraction}, and \eqref{eq:curv-scalar-double-contraction} prove uniqueness.
The Young and trace relations identify $H$ with $(k,2)_0$ and the
other 2 summands with $(k-1,1)_0$ and $(k-2)_0$. In our dimension
range these are distinct irreducible representations and are therefore
orthogonal. The formulas \eqref{eq:curv-vector-embedding} and \eqref{eq:curv-scalar-embedding} use only the metric and polynomial
differentiation, so the embeddings are parallel.
\end{proof}

\begin{proposition}\label{pr:degree-bound}
Let $n\geq8$ and let $(M,g)$ be closed with
$-1\leq\Sec_g\leq-\delta$, $\delta>1/4$.
Every smooth even flow-invariant orthogonal projector in $\Sym^2N$
with finite Fourier content has degree at most 4.
\end{proposition}

\begin{proof}
Suppose that the highest nonzero degree is an even integer $k\geq6$.
By Lemmas~\ref{pr:admissibility} and~\ref{pr:three-components} and \eqref{eq:curv-three-component-splitting},
\begin{equation}\label{eq:curv-top-component-equation}
S=P_k=H+\mathcal J(\eta)+\mathcal I(\tau),\qquad X_+S=0.
\end{equation}
We show successively that $H$, $\eta$ and $\tau$ vanish.

\medskip
\noindent\emph{Step 1: the symmetric component.}
The vector tensor product rule in \cite[Section~2]{CGH00} gives the
following possible types, with non-dominant diagrams omitted:
\begin{equation}\label{eq:curv-tensor-product-types}
\begin{aligned}
V\otimes(k,2)_0:\;&
(k+1,2)_0,(k-1,2)_0,(k,3)_0,(k,1)_0,(k,2,1)_0,\\
V\otimes(k-1,1)_0:\;&
(k,1)_0,(k-2,1)_0,(k-1,2)_0,(k-1)_0,(k-1,1,1)_0,\\
V\otimes(k-2)_0:\;&
(k-1)_0,(k-3)_0,(k-2,1)_0.
\end{aligned}
\end{equation}
Every summand has multiplicity 1. The additional summand with the
same highest weight allowed in odd dimension is absent because the last
highest-weight coordinate is 0.
In particular, by \eqref{eq:curv-tensor-product-types}, $(k+1,2)_0$ can receive a contribution only from
$\nabla H$.

The symbol followed by harmonic projection is nonzero on this summand. To see this,
complexify and choose nonzero mutually orthogonal null vectors $a,b$
in an isotropic highest-weight flag. Set
$$
c(y)=(a\cdot y)b-(b\cdot y)a,
\qquad
H_0(y)=(a\cdot y)^{k-2}c(y)c(y)^{\mathsf T}.
$$
This is harmonic, trace-free and normal, and the symbol
$a\otimes H_0$ has nonzero harmonic image $(a\cdot y)H_0$.
Irreducibility therefore makes the symbol injective on
$(k+1,2)_0$. The equation $X_+S=0$ in \eqref{eq:curv-top-component-equation} gives the hypothesis of
Lemma~\ref{pr:tensor-twistor}. Combining \eqref{eq:diff-tensor-bound} with
\eqref{eq:curv-high-degree-bound} for $k\geq8$, or with
\eqref{eq:curv-sextic-bound} for $k=6$, gives $H=0$.

\medskip
\noindent\emph{Step 2: the vector component.}
After $H=0$, only $\nabla\eta$ can contribute to $(k,1)_0$.
For the corresponding symbol take
$$
\eta_0(y)=(a\cdot y)^{k-2}c(y),\qquad
\eta'_0(y)=(a\cdot y)\eta_0(y).
$$
Modulo $(q)$, the image of $a\otimes\eta_0$ under the symbol of
$X\circ\mathcal J$ is
$y(\eta'_0)^{\mathsf T}+\eta'_0y^{\mathsf T}$.
Its highest harmonic representative is
\begin{equation}\label{eq:curv-vector-symbol-representative}
y(\eta'_0)^{\mathsf T}+\eta'_0y^{\mathsf T}
-\frac q{n+2k-2}D\eta'_0.
\end{equation}
The contraction of \eqref{eq:curv-vector-symbol-representative} with $y$ is a nonzero multiple of $q\eta'_0$.
The symbol is therefore nonzero and injective on the indicated
irreducible summand. Thus the normal harmonic degree-$k$ projection
of $\nabla_y\eta$ vanishes. Since $\eta$ has degree $k-1\geq5$,
\eqref{eq:diff-vector-bound} and \eqref{eq:diff-vector-coercivity} imply $\eta=0$.

\medskip
\noindent\emph{Step 3: the scalar component.}
Only $\mathcal I(\tau)$ remains. The summand $(k-1)_0$ is the
highest symmetric derivative of $\tau$. For
$\tau_0=(a\cdot y)^{k-2}$, the symbol $a\otimes\tau_0$ has image,
modulo $(q)$,
$$
yy^{\mathsf T}(a\cdot y)^{k-1}.
$$
It is not divisible by $q$, so its highest harmonic projection is
nonzero. The symbol is again injective, and the equation
$X_+\mathcal I(\tau)=0$ implies $X_+\tau=0$.
Since $k-2>0$, Lemma~\ref{pr:scalar-vanishing} gives $\tau=0$.
This contradicts $S\ne0$ and proves the proposition.
\end{proof}

\section{Polynomial projectors and curvature}
\label{pr:algebraic-section}

By Proposition~\ref{pr:degree-bound}, the remaining even orthogonal projectors have
Fourier degree at most 4. We first prove that the degree of a nonzero
polynomial normal projector is congruent to 2 modulo 4. We then establish a
curvature estimate for quadratic projectors. The passage from these
fibrewise statements to smooth invariant sections will be made in
Section~\ref{pr:exclusion-section}.

\subsection{The degree of a polynomial projector}
\label{pr:parity-section}

Throughout this subsection, let
$$
S=\R[y_1,\ldots,y_n],\qquad F=\Frac(S),\qquad
q=y_1^2+\cdots+y_n^2,\qquad n\geq3.
$$
The polynomial $q$ is irreducible. The localization $S_{(q)}$ is therefore
a discrete valuation ring; we denote its valuation by $v_q$. Degrees are
computed with $\deg y_i=1$.

\begin{theorem}\label{pr:parity}
Let $P\ne0$ be a symmetric idempotent matrix whose entries belong to
$S[1/q]$ and are homogeneous of degree 0. Assume that $Py=0$. If $m$
is the smallest nonnegative integer such that $q^mP$ is polynomial, then
$m$ is odd.
Consequently, the least polynomial degree of a nonzero even normal
orthogonal projector on $S^{n-1}$ is congruent to 2 modulo 4.
\end{theorem}

The proof has 3 steps. We determine the Witt class of the image of
$P$, identify the quadratic form when its rank is sufficiently small,
and then compute the pole order from an integral Gram matrix.
Orthogonal complementation reduces the general case to the smaller
rank.

\subsubsection{Witt groups and the image quadratic form}

For a regular scheme $X$ and a line bundle $L$, write $W^0(X,L)$ for the
Witt group of nondegenerate symmetric bundles with values in $L$, and
$W^0(X)=W^0(X,\mathcal O_X)$. For a field $K$, write $W(K)$ for its
Witt group, and let $\langle a\rangle$ denote the 1-dimensional form
$az^2$. Integer multiples below are multiples in the
Witt group and may have negative coefficients.

We recall the facts about Witt groups that will be used. The schemes in
the localization argument are regular and quasi-projective over $\R$,
with $2$ invertible.
Their usual and derived degree-0 Witt groups agree, including for
line bundle-valued dualities; see \cite[Theorems~4.3 and~4.7, Remark~4.9]{Balmer01}. Arason's
projective-space theorem \cite{Arason80} gives
$$
W(\R)\simeq W^0(\mathbb P^{n-1}_\R).
$$
If $D=\operatorname{Proj}S/(q)$, structural pullback also gives a
surjection $W(\R)\to W^0(D)$ by
\cite[Theorems~8.9--8.11]{Xie19}.\footnote{For $\dim D\equiv0\pmod4$, the extra term in
Theorem~8.10 is a subgroup of $\Z$. It vanishes because no residue field
of $D$ admits an ordering, so $W^0(D)$ is $2$-primary torsion by
\cite[Proposition~3.1]{DellAmbrogioFasel08}.}
In particular, both groups are generated by the unit form.

For a regular divisor $i:D\hookrightarrow X$ in this setting, with
normal line bundle $N_{D/X}$ and complement $U$, we have an exact sequence
$$
W^0(X)\longrightarrow W^0(U)\xrightarrow{\partial}
W^0(D,N_{D/X});
$$
see \cite{Balmer00,Gille07}. For a vector bundle $E$ on $D$, the
coefficient line is fixed by the sheaf-duality identification
$$
\mathcal{E}xt^1_{\mathcal O_X}(i_*E,\mathcal O_X)
=i_*(E^*\otimes N_{D/X}).
$$
Tensoring a symmetric bundle by a line bundle changes its coefficient
line by the square of that line bundle.

\begin{lemma}\label{pr:witt-generators}
Let $X=\mathbb P^{n-1}_\R$ and $U=X\setminus D$. Equip the tautological
line $\ell=\mathcal O_X(-1)|_U$ with its Euclidean form. Then $W^0(U)$ is
generated by the unit form and $\ell$.
\end{lemma}
\begin{proof}
Here $N_{D/X}=\mathcal O_D(2)$. Tensoring by $\mathcal O_D(-1)$ identifies
$W^0(D,\mathcal O_D(2))$ with $W^0(D)$, which is generated by the unit.
The form on $\ell$ extends to the symmetric morphism
$$
q:\mathcal O_X(-1)\longrightarrow\mathcal O_X(1).
$$
Its cokernel has the resolution
$$
0\longrightarrow\mathcal O_X(-1)\xrightarrow{q}\mathcal O_X(1)
\longrightarrow i_*\mathcal O_D(1)\longrightarrow0.
$$
The symmetric-cone construction of the localization boundary
\cite[Lemma~4.2]{BalmerCalmes09} therefore gives, up to sign, the
multiplication pairing on $\mathcal O_D(1)$ with values in
$\mathcal O_D(2)$. Under the
preceding identification this is the unit form. Hence $\partial[\ell]$
generates the localization target. Subtracting a multiple of $[\ell]$
from any element of $W^0(U)$ leaves an element in the image of $W^0(X)$,
and this image is generated by the unit.
\end{proof}

\begin{lemma}\label{pr:witt-expression}
Let $r=\rank P$ and equip $E_F=\im(P:F^n\to F^n)$ with the restriction of
the Euclidean form. There is an integer $b$ such that
\begin{equation}\label{eq:alg-witt-class}
[E_F]=(r-b)\langle1\rangle+b\langle q\rangle
\qquad\text{in }W(F).
\end{equation}
\end{lemma}
\begin{proof}
The degree-0 part of $S[1/q]$ is the coordinate ring of $U=D_+(q)$.
Thus $P$ defines a self-adjoint idempotent on the trivial bundle over
$U$. Its image is a quadratic bundle $E$, and
Lemma~\ref{pr:witt-generators} gives $[E]=a[1]+b[\ell]$. At a real point of
$U$, all 3 forms are positive definite; comparison of signatures
gives $a+b=r$. On the chart $y_1\ne0$, the form on $\ell$ is
$\langle q/y_1^2\rangle$. Its extension to $F$ is $\langle q\rangle$,
which proves \eqref{eq:alg-witt-class}.
\end{proof}

\subsubsection{Residues and the smaller-rank case}

We use the standard facts that a Pfister form is round, that an isotropic
Pfister form is hyperbolic, and that anisotropic forms are determined by
their Witt classes; see \cite{Pfister66,Lam2005}. We also use Springer's
theorem over a complete discretely valued field of residue
characteristic different from 2: the anisotropic dimension is the sum
of the anisotropic dimensions of its 2 residue forms; see
\cite[Theorem~10 and Corollary~11]{ET11}.

\begin{lemma}\label{pr:generic-squares}
For $N\geq1$, the polynomial $y_1^2+\cdots+y_N^2$ is not a sum of fewer
than $N$ squares in $\R(y_1,\ldots,y_N)$.
\end{lemma}
\begin{proof}
We argue by induction on $N$. The assertion for $N=1$ is immediate.
Suppose that
$$
y_1^2+Q=\sum_{i=1}^m f_i^2,\qquad
Q=y_2^2+\cdots+y_N^2,\qquad 1\leq m<N.
$$
Over $K=\R(y_2,\ldots,y_N)$, the Cassels--Pfister theorem
\cite{Cassels64}, also \cite[p.~3, Corollary~2]{Gerstein08}, replaces the
$f_i\in K(y_1)$ by polynomials in $K[y_1]$. Since $K$ is formally real,
leading squares cannot cancel. Each of these polynomials has degree at
most 1, so comparison of coefficients gives
$$
f_i=a_i y_1+b_i,\qquad
\sum_i a_i^2=1,\qquad \sum_i a_ib_i=0,\qquad \sum_i b_i^2=Q.
$$
By Witt's extension theorem, an orthogonal transformation sends
$(a_i)$ to the 1st standard basis vector. It sends $(b_i)$ to a vector
with 1st coordinate 0. Thus $Q$ is a sum of $m-1$ squares over $K$,
contrary to the induction hypothesis.
\end{proof}

\begin{lemma}\label{pr:constant-residues}
Let
$$
p=2^{\lfloor\log_2(n-1)\rfloor},\qquad p<n\leq2p,\qquad
\varkappa=\Frac(S/(q)).
$$
Over $\varkappa$, the form $p\langle1\rangle$ is anisotropic and
$2p\langle1\rangle$ is hyperbolic. The anisotropic dimension of
$a\langle1\rangle$, for $a\in\Z$, is
$$
d(a)=\min_{j\in\Z}|a-2pj|.
$$
\end{lemma}
\begin{proof}
Set $K=\R(y_1,\ldots,y_{n-1})$ and
$Q=y_1^2+\cdots+y_{n-1}^2$. Then
$\varkappa=K(\sqrt{-Q})$. If the Pfister form
$\pi=p\langle1\rangle$ were isotropic over $\varkappa$, a nonzero
isotropic vector $u+\sqrt{-Q}\,v$ would satisfy
$$
\langle u,v\rangle=0,\qquad |u|^2=Q|v|^2,
\qquad u,v\in K^p.
$$
Formal reality implies $v\ne0$. Put $\alpha=|v|^2$. Roundness gives
$\pi\simeq\alpha\pi$, and Witt's extension theorem allows the isometry
to send $v$ to the 1st standard basis vector. The image of $u$ has
1st coordinate 0 and norm $Q\alpha$ for the form $\alpha\pi$.
Consequently, $Q$ is a sum of $p-1$ squares in $K$, contradicting
Lemma~\ref{pr:generic-squares}. This proves anisotropy.

The equation $q=0$ makes $n\langle1\rangle$ isotropic over $\varkappa$.
Since $n\leq2p$, the Pfister form $2p\langle1\rangle$ is isotropic and
hence hyperbolic. Reduce $a$ modulo $2p$ to $c\in[-p,p]$. Its Witt class
is represented by $|c|$ copies of $\langle1\rangle$ or
$\langle-1\rangle$, according to the sign of $c$. This form is
anisotropic, being a subform of $\pm p\langle1\rangle$, and has dimension
$d(a)$.
\end{proof}

\begin{lemma}\label{pr:small-rank-form}
If $1\leq r<p$, then
\begin{equation}\label{eq:alg-small-rank-form}
E_F\simeq q\,r\langle1\rangle.
\end{equation}
\end{lemma}
\begin{proof}
Complete $S_{(q)}$, whose residue field is $\varkappa$. By Springer's
theorem and Lemma~\ref{pr:constant-residues}, the anisotropic dimension of the
class in \eqref{eq:alg-witt-class} over the completed fraction field is
$d(r-b)+d(b)$. Since this class has a rank-$r$ representative,
$$
d(r-b)+d(b)\leq r.
$$
Choose shortest representatives $s,t$ of $b,r-b$ modulo $2p$.
Then $|s|+|t|\leq r<p$ and $s+t\equiv r\pmod{2p}$. Hence $s+t=r$, and
$|s|+|t|=s+t$ forces $s,t\geq0$. It follows that
$b\equiv b_0\pmod{2p}$ for some $0\leq b_0\leq r$.

The Pfister form $2p\langle1\rangle$ represents $q$ over $F$, because
$n\leq2p$. Roundness therefore gives
$2p\langle q\rangle\simeq2p\langle1\rangle$. We may consequently replace
$b$ by $b_0$ in the Witt-class formula:
$$
[E_F]=(r-b_0)\langle1\rangle+b_0\langle q\rangle.
$$
Both sides have anisotropic rank-$r$ representatives over the formally
real field $F$, so the equality is an isometry.

If $r-b_0>0$, there is a unit vector $e\in E_F\subset y^\perp$.
The orthogonal vectors $e,y$ give an embedding
$\langle1,q\rangle\hookrightarrow n\langle1\rangle$. Cancelling
$\langle1\rangle$ would make $q$ a sum of $n-1$ rational squares,
contradicting Lemma~\ref{pr:generic-squares}. Thus $b_0=r$.
\end{proof}

\subsubsection{The pole order}

\begin{lemma}\label{pr:odd-gram}
If $1\leq r<p$, the smallest exponent $m$ for which $q^mP$ is polynomial
is odd.
\end{lemma}
\begin{proof}
Use the coordinates of the isometry in \eqref{eq:alg-small-rank-form}.
For $0\ne z\in E_F$, write its coordinates as $c_1,\ldots,c_r$.
Anisotropy of $r\langle1\rangle$ over $\varkappa$ prevents cancellation
of the leading residue, so
$$
v_q(\langle z,z\rangle)=1+2\min_i v_q(c_i).
$$
In particular, every nonzero norm in $E_F$ has odd valuation.

Put $\mathcal O=S_{(q)}$ and $L=E_F\cap\mathcal O^n$. This is a saturated
rank-$r$ lattice. Its inclusion in $\mathcal O^n$ splits, because the
quotient is a finite torsion-free module over a discrete valuation ring.
Let $A$ be a basis matrix for $L$ and choose an integral left inverse
$C$, so $CA=\Id_r$.

The Gram matrix $G=A^{\mathsf T}A$ is integral and nondegenerate. Since
$2$ is a unit, it can be diagonalized by an integral congruence:
\begin{equation}\label{eq:alg-gram-diagonal}
G\sim\operatorname{diag}(q^{a_1}u_1,\ldots,q^{a_r}u_r),
\qquad a_i\geq0,\quad u_i\in\mathcal O^\times.
\end{equation}
Indeed, a matrix entry of least valuation can be made a diagonal pivot:
if the least valuation occurs only off the diagonal, replace the 2
corresponding basis vectors by their sum and one of them. The resulting
diagonal entry has the least valuation. The elimination coefficients
lie in $\mathcal O$, and induction completes the diagonalization. Each
$q^{a_i}u_i$ in \eqref{eq:alg-gram-diagonal} is a nonzero norm in $L$, so every $a_i$ is odd.

The orthogonal projector onto $E_F$ is
\begin{equation}\label{eq:alg-projector-gram}
P=AG^{-1}A^{\mathsf T},\qquad CPC^{\mathsf T}=G^{-1}.
\end{equation}
Because $A,C$ are integral, the 2 identities in \eqref{eq:alg-projector-gram} show that the exact
$q$-pole order of $P$ is $\max_i a_i$, which is odd. Its entries have no
other denominator primes, so this is also the least exponent for which
$q^mP$ is polynomial.
\end{proof}

\begin{proof}[Proof of Theorem~\ref{pr:parity}]
If $m=0$, homogeneity makes $P$ constant, and $Py=0$ forces $P=0$.
Thus $m\geq1$. For $r=n-1$, the projector is
$\Id-yy^{\mathsf T}/q$, with pole order 1.

Otherwise its normal orthogonal complement is the nonzero projector
\begin{equation}\label{eq:alg-normal-complement}
P'=\Id-yy^{\mathsf T}/q-P.
\end{equation}
Its rank is $n-1-r$, and
$$
\min\{r,n-1-r\}\leq\lfloor(n-1)/2\rfloor<p.
$$
If $m\geq2$, subtracting a matrix with pole order at most 1 in \eqref{eq:alg-normal-complement} cannot
cancel a pole of order $m$, so $P'$ has the same pole order as $P$.
Lemma~\ref{pr:odd-gram}, applied to whichever projector has smaller rank,
proves the assertion. The case $m=1$ is immediate.

Finally, $q^mP$ is homogeneous of degree $2m$ and represents $P$ on the
sphere. A lower-degree representative can first be made even by
averaging under $y\mapsto-y$, and then homogenized with powers of $q$
to degree at most $2m-2$. Equality on the sphere and
radial homogeneity would then imply that $q^{m-1}P$ is polynomial.
This contradicts the definition of $m$. Hence the least degree is
$2m\equiv2\pmod4$.
\end{proof}

\subsection{Quadratic projectors}
\label{pr:quadratic-section}

Let $V$ be an $n$-dimensional Euclidean space, $n\geq3$, and write
$q(y)=|y|^2$. We now prove the quadratic estimate.

\begin{theorem}\label{pr:quadratic}
Let $Q:V\to\Sym^2V$ be homogeneous quadratic and satisfy
\begin{equation}\label{eq:quad-projector-relations}
Q(y)y=0,\qquad Q(y)^2=q(y)Q(y),\qquad
\tr Q(y)=rq(y),\qquad 0\leq r\leq n-1.
\end{equation}
For $v\in S(V)$, put $P_v=Q(v)$. If $A$ is an algebraic curvature
tensor with $\delta\leq\Sec_A\leq1$ and $\delta\geq1/4$, then
\begin{equation}\label{eq:quad-curvature-estimate}
W_A^{(2)}(P)\geq\frac23(4\delta-1)r(n-1-r).
\end{equation}
\end{theorem}

We first describe the quadratic coefficients and their consequences for
the vertical derivative. A 4th-moment computation will then express
the curvature term using the Jacobi operator alone.

\subsubsection{Complementation and the quadratic coefficients}

\begin{lemma}\label{pr:complement}
Let $T_v=\Id-vv^{\mathsf T}$ and $P'_v=T_v-P_v$. Then $P'$ is a
quadratic normal orthogonal projector of rank $n-1-r$, and
$W_A^{(2)}(P')=W_A^{(2)}(P)$ for every algebraic curvature tensor $A$.
\end{lemma}
\begin{proof}
Its quadratic numerator is $q\Id-yy^{\mathsf T}-Q$. At a fixed unit
vector $v$, choose the orthogonal splitting
\begin{equation}\label{eq:quad-orthogonal-splitting}
V=\R v\oplus E_v\oplus F_v,\qquad
E_v=\im P_v,\quad F_v=\im P'_v.
\end{equation}
Use indices $\alpha$ for $E_v$ and $p$ for $F_v$, with $e_0=v$.
Keep this output basis fixed when taking derivatives, and set
$z_{ab}=\nabla_SP_{ab}$. Differentiating $P^2=P$ in \eqref{eq:geo-projector-identities} shows
that only the blocks between $E_v$ and its orthogonal complement can
have nonzero derivative. Differentiating $Pv=0$ gives
\begin{equation}\label{eq:quad-radial-blocks}
z_{\alpha0}=-e_\alpha,\qquad z_{p0}=0.
\end{equation}
The terms containing a radial output index cancel in
\eqref{pr:W2}. The remaining density is
\begin{equation}\label{eq:quad-mixed-density}
2\sum_{\alpha,p}\bigl\{
 A(v,z_{\alpha p},v,z_{\alpha p})
 -A(e_\alpha,e_p,v,z_{\alpha p})\bigr\}.
\end{equation}
The $E_v$--$F_v$ block of $dT$ vanishes. Hence the corresponding
gradients for $P'$ are $z'_{p\alpha}=-z_{\alpha p}$. Interchanging
$E_v$ and $F_v$ in \eqref{eq:quad-mixed-density} leaves both terms unchanged.
\end{proof}

The rank-0 and rank-$(n-1)$ projectors have no mixed blocks, and their
curvature functional is 0. For all other ranks, complementation
allows us to assume
\begin{equation}\label{eq:quad-reduced-rank}
1\leq r<n/2,\qquad s=n-1-r.
\end{equation}

\begin{lemma}\label{pr:quadratic-coefficients}
Write $Q_{ab}(y)=\sum_{i,j}B_{abij}y_iy_j$, using coefficients
symmetric in $i,j$ and placing output indices first.
Then
\begin{equation}\label{eq:quad-coefficient-relations}
\begin{gathered}
B_{abij}=B_{baij}=B_{abji}=B_{ijab},\\
B_{abij}+B_{aibj}+B_{ajbi}=0,\qquad B(\Id)=r\Id.
\end{gathered}
\end{equation}
Here $B$ acts on symmetric matrices by contraction in the last pair.
It is self-adjoint. For unit vectors $u,v$,
\begin{equation}\label{eq:quad-image-reciprocity}
u\in\im P_v\quad\Longleftrightarrow\quad v\in\im P_u.
\end{equation}
\end{lemma}
\begin{proof}
This is the quadratic case of the normality computation in
\cite[Lemma~4.2]{CLMS24}. For completeness, let $B(x,z)$ be the
polarization of $Q$ and set
$C(x,z;y,w)=\langle B(x,z)y,w\rangle$. It is symmetric within each
pair. Polarizing $Q(x)x=0$ in \eqref{eq:quad-projector-relations} gives
$$
C(x,z;y,w)+C(z,y;x,w)+C(y,x;z,w)=0.
$$
Add this identity with $(x,z,y;w)$ and $(x,z,w;y)$, and subtract the
identities with $(y,w,x;z)$ and $(y,w,z;x)$. The within-pair symmetries
cancel the other terms and give
$C(x,z;y,w)=C(y,w;x,z)$. This proves pair exchange and the
cyclic identity in \eqref{eq:quad-coefficient-relations}. Pair exchange also converts $\tr Q(y)=r|y|^2$ into
$B(\Id)=r\Id$.

In particular, $\langle P_uv,v\rangle=\langle P_vu,u\rangle$.
For an orthogonal projector, its quadratic form on a unit vector equals
1 exactly when that vector is in its image. This proves \eqref{eq:quad-image-reciprocity}.
\end{proof}

\begin{lemma}\label{pr:pencil}
Let $v,w$ be orthonormal and define, on $E_v$,
$$
H_{vw}=\Id_{E_v}-(P_vP_wP_v)|_{E_v}.
$$
If $r$ is odd, then $\tr(P_vP_w)\leq r-1$. If $r$ is even, every
eigenvalue of $H_{vw}$ has even multiplicity, and
$$
r-\tr(P_vP_w)\geq2|P_vw|^2.
$$
\end{lemma}
\begin{proof}
For fixed $v$, the condition $E_v\cap E_w=0$ holds for a dense set of
$w\in S(v^\perp)$. Indeed, if $u\in S(E_v\cap E_w)$, \eqref{eq:quad-image-reciprocity}
gives $v,w\in E_u$. This is impossible for $r=1$. For $r\geq2$,
all such $w$ lie in the image of the smooth sphere bundle
$$
u\in S(E_v),\qquad w\in S(E_u\cap v^\perp).
$$
This bundle has dimension $(r-1)+(r-2)=2r-3<n-2$; its image has empty
interior. The set where the intersection is nonzero is closed, as a
rank condition, proving the density claim.

For a pair with 0 intersection, put $A_0=P_v$, $C_0=P_w$ and
$F=2B(v,w)$. Comparing coefficients in
$Q(xv+yw)^2=(x^2+y^2)Q(xv+yw)$ gives
\begin{equation}\label{eq:quad-pencil-identities}
A_0F+FA_0=F,\qquad C_0F+FC_0=F,\qquad
F^2=(A_0-C_0)^2.
\end{equation}
Thus $D=A_0-C_0$, $E=A_0+C_0-\Id$ and $F$ are pairwise anticommuting
self-adjoint operators, with
$$
D^2=F^2,\qquad D^2+E^2=\Id.
$$
Decompose into the eigenspaces of $D^2$. On a block with eigenvalue
$0<\lambda<1$, the operator $DF/\lambda$ is a complex structure
commuting with $E$. Both eigenspaces of $E$ therefore have even
dimension, and $D$ exchanges them. The block has dimension divisible
by 4. Since $F$ is invertible and anticommutes with
$D+E=2A_0-\Id$, it exchanges the 2 eigenspaces of $A_0$.
The image of $A_0$ on this block consequently has even dimension.

On the 0 block, $A_0=C_0$, and the absence of a common image forces
both to vanish. On the block with $\lambda=1$, $E=0$ and
$$
A_0=(\Id+D)/2,\qquad C_0=(\Id-D)/2,
$$
so $A_0C_0=0$. Also
$D^2A_0=A_0-A_0C_0A_0$, and hence $H_{vw}=D^2|_{E_v}$.
For odd $r$, some image direction must occur in the block with
$\lambda=1$, giving $\tr(P_vP_w)\leq r-1$. For even $r$, all
intermediate eigenvalue multiplicities are even, and the multiplicity
of 1 is then even as well. Continuity of the ordered eigenvalues
extends these conclusions to every orthonormal pair.

Finally, $P_ww=0$ gives $P_w\leq\Id-ww^{\mathsf T}$, so
$$
H_{vw}\geq(P_vw)\otimes(P_vw).
$$
The largest eigenvalue is at least $|P_vw|^2$ and, for even $r$, occurs
at least 2 times. Taking traces proves the last inequality.
\end{proof}

\subsubsection{The vertical derivative and the curvature term}

\begin{lemma}\label{pr:free-gradient}
In the splitting \eqref{eq:quad-orthogonal-splitting}, let
$$
Z_v=\sum_{\alpha,p}z_{\alpha p}\otimes z_{\alpha p}
\quad\text{on }v^\perp.
$$
Then
\begin{equation}\label{eq:quad-gradient-identities}
Z_v=r\Id-P_v-B(P_v),\qquad
0\leq Z_v\leq r\Id-P_v,\qquad \tr Z_v=rs.
\end{equation}
Moreover, $Z_v+P_v\geq\Id$ for odd $r$, and $Z_v\geq P_v$ for even
$r$, as inequalities on $v^\perp$.
\end{lemma}
\begin{proof}
For $w\perp v$, the radial blocks in \eqref{eq:quad-radial-blocks}
give, in the Hilbert--Schmidt norm,
$$
|dP_v(w)|_{\mathrm{HS}}^2
=2|P_vw|^2+2\langle Z_vw,w\rangle.
$$
When $|w|=1$, the coefficient identities in
\eqref{eq:quad-pencil-identities} give
$dP_v(w)^2=(P_v-P_w)^2$. Taking traces yields
\begin{equation}\label{pr:free-overlap}
\langle(Z_v+P_v)w,w\rangle=r-\tr(P_vP_w).
\end{equation}
Pair exchange gives
$$
\tr(P_vP_w)=\langle B(P_v)w,w\rangle,\qquad
B(P_v)=\sum_\alpha Q(e_\alpha)\geq0.
$$
The operator identity and upper bound in \eqref{eq:quad-gradient-identities} follow, and Lemma~\ref{pr:pencil}
gives the 2 lower bounds.

Quadraticity gives $\Delta_SP=2r\Id-2nP$.
Choose an orthonormal basis $(e_i)_{i=1}^{n-1}$ of $v^\perp$.
Applying this Laplacian to $P^2=P$ in \eqref{eq:geo-projector-identities} gives
$$
\sum_i(dP(e_i))^2=r\Id+(n-2r)P.
$$
Thus $|dP|^2=2r(n-r)$. The radial output blocks contribute $2r$,
while the mixed blocks contribute $2\tr Z_v$. Hence $\tr Z_v=rs$.
\end{proof}

\begin{lemma}\label{pr:fourth-moment}
For every algebraic curvature tensor $A$,
\begin{equation}\label{eq:quad-fourth-moment-curvature}
W_A^{(2)}(P)=\frac23\int_{S(V)}
\tr\bigl(A_v[5Z_v+(2r+1-n)P_v-r\Id_{v^\perp}]\bigr)\,d\sigma.
\end{equation}
\end{lemma}
\begin{proof}
Write the Jacobi polynomial using coefficients symmetric in $i,j$:
$$
\langle A_ve_b,e_a\rangle=\sum_{i,j}D_{abij}v_iv_j.
$$
The coefficients $D$ have the same pair and cyclic symmetries as $B$,
and act by the same contraction convention, with $D(\Id)=\Ric_A$.
The curvature tensor is recovered by
\begin{equation}\label{eq:quad-curvature-reconstruction}
A_{caij}=\frac23(D_{ciaj}-D_{aicj}),\qquad
A_{caij}=A(e_c,e_a,e_i,e_j).
\end{equation}
All sums below use full ordered indices; symmetric matrix entries are
not counted with an extra factor. Put $c=n+4-2r$ and define
$$
H_{acij}=\sum_{b,\ell}B_{abi\ell}B_{cbj\ell},\qquad
(B^2)_{acij}=\sum_{b,\ell}B_{acb\ell}B_{b\ell ij}.
$$
The 2nd expression is composition on $\Sym^2V$, rather than the
pointwise square of $Q$. Pair exchange gives
$H_{caij}=H_{acji}$ and $H_{ijac}=H_{acij}$. Applying the cyclic
identity to the 2 factors of $B^2$ and relabelling gives
$$
(B^2)_{acij}=2(H_{aicj}+H_{ajci}).
$$
On the other hand, applying the Euclidean Laplacian to $Q^2=qQ$
gives
$$
4\sum_\ell B(y,e_\ell)^2=cQ(y)+rq(y)\Id,
$$
and therefore
$$
\frac{H_{acij}+H_{caij}}2
=\frac14(cB_{acij}+r\delta_{ac}\delta_{ij}).
$$
Use the full coefficient contraction $\langle D,B\rangle
=\sum D_{abij}B_{abij}$, and likewise for $B^2$. The preceding
identities and the cyclic relation for $D$ give
$$
\begin{aligned}
\sum D_{aicj}H_{acij}&=\frac14\langle D,B^2\rangle,\\
\sum D_{aicj}(H_{acij}+H_{caij})
&=-\frac14\{c\langle D,B\rangle+r\scal_A\}.
\end{aligned}
$$
Substituting in the curvature reconstruction formula \eqref{eq:quad-curvature-reconstruction} yields
$$
\sum A_{caij}H_{caij}
=\frac13\langle D,B^2\rangle
 +\frac c6\langle D,B\rangle+\frac r6\scal_A.
$$

The normalized sphere measure satisfies
\begin{equation}\label{eq:quad-sphere-fourth-moment}
\int v_iv_jv_kv_\ell\,d\sigma
=\frac{\delta_{ij}\delta_{k\ell}+\delta_{ik}\delta_{j\ell}
+\delta_{i\ell}\delta_{jk}}{n(n+2)}.
\end{equation}
Let
$$
\begin{aligned}
J&=\int\tr(A_vP_v)\,d\sigma,\\
J_2&=\int\tr(A_vB(P_v))\,d\sigma,\\
C&=\int\sum_{a,b}A(P_ve_b,e_a,v,z_{ab})\,d\sigma.
\end{aligned}
$$
The 4th-moment formula \eqref{eq:quad-sphere-fourth-moment} gives
$$
\begin{aligned}
 n(n+2)J&=r\scal_A+2\langle D,B\rangle,\\
 n(n+2)J_2&=r^2\scal_A+2\langle D,B^2\rangle.
\end{aligned}
$$
For the twisting term, substitute
$z_{ab}=2B_{ab}v-2P_{ab}v$, where $(B_{ab})_{ij}=B_{abij}$.
The radial summand drops out of $v\wedge z_{ab}$. 1 4th-moment
contraction vanishes by $\sum_a B_{aaij}=r\delta_{ij}$ and the antisymmetry of the
1st curvature pair; the other 2 agree. Hence
$$
C=\frac4{n(n+2)}\sum A_{caij}H_{caij},\qquad
3C=cJ+2J_2-r\frac{\scal_A}n.
$$
The identity for $Z_v$ in \eqref{eq:quad-gradient-identities} gives
$$
\int\tr(A_vZ_v)\,d\sigma
=r\frac{\scal_A}n-J-J_2.
$$
The Jacobi term in \eqref{pr:W2} is
$2\int\tr(A_v(Z_v+P_v))\,d\sigma$. Combining these equalities proves
\eqref{eq:quad-fourth-moment-curvature}.
\end{proof}

\subsubsection{Proof of the estimate}

\begin{proof}[Proof of Theorem~\ref{pr:quadratic}]
Assume $1\leq r<n/2$ (see \eqref{eq:quad-reduced-rank}), as permitted by Lemma~\ref{pr:complement}, and put
$K=n-1+3r$. On $v^\perp$, consider
$$
S_v=5r\Id+(n-1-2r)P_v-5Z_v.
$$
Since $n-1-2r\geq0$, the upper bound for $Z_v$ gives
$S_v\geq(n+4-2r)P_v\geq0$. For odd $r$, the lower bound gives
$$
S_v\leq(5r-5)\Id+(n+4-2r)P_v\leq K\Id.
$$
For even $r$, it gives
$$
S_v\leq5r\Id+(n-6-2r)P_v
\leq\max\{5r,n-6+3r\}\Id\leq K\Id.
$$
The trace identity in \eqref{eq:quad-gradient-identities} consequently yields
\begin{equation}\label{eq:quad-estimate-traces}
0\leq S_v\leq K\Id,\qquad \tr S_v=rK,
\qquad\tr(4r\Id-S_v)=3rs.
\end{equation}
For a self-adjoint operator $T$, write $T_+$ for its positive spectral
part and $T_-$ for the absolute value of its negative spectral part;
the same convention applies to scalars.
Because $K\geq4r$, for $0\leq\lambda\leq K$ we have
$(\lambda-4r)_+\leq(1-4r/K)\lambda$. Thus
$$
\tr(4r\Id-S_v)_-
=\tr(S_v-4r\Id)_+\leq rs.
$$
The pinching assumption gives $\delta\Id\leq A_v\leq\Id$. Applying
these inequalities and \eqref{eq:quad-estimate-traces} to the positive and negative spectral parts of
$4r\Id-S_v$, without any commutation assumption, gives
$$
\begin{aligned}
\tr(A_v(4r\Id-S_v))
&\geq\delta\tr(4r\Id-S_v)_+-\tr(4r\Id-S_v)_-\\
&=3\delta rs-(1-\delta)\tr(4r\Id-S_v)_-\\
&\geq(4\delta-1)rs.
\end{aligned}
$$
Integrating and applying \eqref{eq:quad-fourth-moment-curvature} proves \eqref{eq:quad-curvature-estimate}.
Complementation handles the other ranks.
\end{proof}

\begin{corollary}\label{pr:quadratic-exclusion}
Let $M^n$ be closed, $n\geq3$, with strictly quarter-pinched negative
sectional curvature. There is no smooth invariant quadratic normal
orthogonal projector of rank $1\leq r\leq n-2$.
\end{corollary}
\begin{proof}
After a constant rescaling of the metric, assume
$-1\leq\Sec_g\leq-\delta$ with $\delta>1/4$.
The identity $B(\Id)=r\Id$ gives the sphere mean
$\int P\,d\sigma=(r/n)\Id$. Thus
\begin{equation}\label{eq:quad-harmonic-decomposition}
P=(r/n)\Id+u_2,
\end{equation}
where $u_2$ has harmonic degree 2; compare
\cite[Lemma~4.2]{CLMS24}. The scalar part in \eqref{eq:quad-harmonic-decomposition} is parallel. Invariance
therefore gives $Xu_2=0$, and the different harmonic degrees of
$X_+u_2$ and $X_-u_2$ imply that both vanish. The twisted Pestov identity
\eqref{eq:diff-pestov} then gives
\begin{equation}\label{eq:quad-pestov-integral}
\int_M W_{-\mathcal R_g}^{(2)}(u_2)\,d\vol\leq0.
\end{equation}
Adding a constant scalar multiple of the identity does not change
$W^{(2)}$: the vertical derivative is unchanged and the additional
twisting contraction pairs a skew curvature index pair with a
symmetric derivative. Hence \eqref{eq:quad-pestov-integral} also holds for $P$.
At each base point, apply \eqref{eq:quad-curvature-estimate} to $A=-\mathcal R_g$. Its lower bound is
strictly positive because $\delta>1/4$ and $r(n-1-r)>0$. This
contradicts \eqref{eq:quad-pestov-integral}.
\end{proof}

\section{Non-existence of invariant normal projectors}
\label{pr:exclusion-section}

We now complete the argument for orthogonal projectors. As in
\cite[Section~4.2]{CLMS24}, the proof first reduces the Fourier degree and
then treats the quadratic term. For rank 1, polynomial factorization
allows us to carry out the reduction in every dimension $n\geq4$.

\subsection{Reduction to degree 2}
\label{pr:closure-section}

\begin{lemma}\label{pr:smooth-descent}
Let $P$ be a smooth family of even normal orthogonal projectors on the
unit spheres of a Euclidean vector bundle of rank $n\geq3$. If $P$ has
degree at most 4 in every fibre, then it has degree at most 2, and
its homogeneous quadratic numerator depends smoothly on the base point.
\end{lemma}

\begin{proof}
In a local orthonormal frame, let $P_j$ denote the homogeneous harmonic
representative of the degree-$j$ component of $P$. Its degree-0
extension away from the zero section is
\begin{equation}\label{eq:proj-degree-four-extension}
  P(y)=\frac{q^2P_0(y)+qP_2(y)+P_4(y)}{q^2},
  \qquad q=|y|^2.
\end{equation}
By Theorem~\ref{pr:parity}, a nonzero normal projector cannot have pole order
2. Thus the numerator in \eqref{eq:proj-degree-four-extension} is divisible by $q$ in each fibre; this is also
true for the 0 projector. Multiplication by $q$ is an injective linear
map between fixed finite-dimensional polynomial spaces. Its inverse on
its image is linear, so the quotient is smooth, including at points
where the degree drops. Since division by $q$ is orthogonally
equivariant, the local quotients define a global quadratic numerator.
\end{proof}

\begin{theorem}\label{pr:closure}
Let $(M^n,g)$ be closed and strictly quarter-pinched with negative
sectional curvature, and suppose that $n\geq8$. There is no smooth,
even, flow-invariant normal orthogonal projector with finite Fourier
content and rank $1\leq r\leq n-2$.
\end{theorem}

\begin{proof}
After a constant rescaling of the metric, assume
$-1\leq\Sec_g\leq-\delta$ with $\delta>1/4$. By
Proposition~\ref{pr:degree-bound}, such a projector has degree at most 4.
Lemma~\ref{pr:smooth-descent} reduces its degree to at most 2, and
Corollary~\ref{pr:quadratic-exclusion} excludes the quadratic case. A normal
projector of degree 0 is 0, since its constant matrix annihilates
every vector in the fibre.
\end{proof}

\subsection{Rank-1 projectors}
\label{rank1:section}

In rank 1 we factor the projector before applying the harmonic
estimates. The fixed homogeneous degree in the next lemma is needed to
obtain a smooth factor when the minimal degree varies with the base
point.

\begin{lemma}[Polynomial factorization]\label{rank1:dyad-lift}
Let $P$ be a smooth family of even rank-1 normal orthogonal projectors
on the unit spheres of a Euclidean vector bundle of rank $n\geq4$.
Assume a uniform finite bound on its polynomial degree. Then there are
an odd integer $D$ and a flat real line bundle such that $P$ admits a
smooth homogeneous vector polynomial $v$, with values in the tensor
product of the Euclidean bundle and that line bundle, satisfying
\begin{equation}\label{eq:rank1-factor-identities}
  q^DP=vv^{\mathsf T},\qquad |v(y)|^2=q(y)^D,
  \qquad \langle v(y),y\rangle=0.
\end{equation}
Here $P$ is extended homogeneously with degree 0. Local factors are
unique up to sign. For a family over $SM$, $XP=0$ implies $Xv=0$.
\end{lemma}

\begin{proof}
Choose a fixed odd integer $D$ large enough that $q^DP$ is polynomial
in every fibre. At a fixed base point, divide a nonzero column of
$q^DP$ by the greatest common divisor of its entries. This gives a
primitive homogeneous vector $z\in\mathbb R[y_1,\ldots,y_n]^n$ spanning
the image of $P$ over the fraction field. Since $P$ is an orthogonal
projector of rank 1,
$$
  P=\frac{zz^{\mathsf T}}{|z|^2}.
$$
For any irreducible polynomial $\pi$ not associated to $q$, some coordinate $z_j$ is
not divisible by $\pi$. The entry $P_{jj}$ has no pole at $\pi$, so
$|z|^2$ has no factor $\pi$. Unique factorization therefore gives
$|z|^2=a q^d$, with $a>0$ and $d\geq0$. After rescaling $z$, we may
assume that $a=1$. Normality gives $z\cdot y=0$, and homogeneity
implies $\deg z=d$. A coordinate not
divisible by $q$ shows that $d$ is the pole order of $P$, so $d\leq D$.
On the unit sphere, $z$ is homotopic to the identity through
$(\cos t)y+(\sin t)z(y)$, $0\leq t\leq\pi/2$, and
therefore has topological degree 1. If $d$ were even, $z$ would factor
through $\mathbb{RP}^{n-1}$; its degree would be 0 for odd $n$ and
even for even $n$, a contradiction
\cite[Section~2.2, Exercise~14]{Hatcher02}. Thus $d$ is odd, and
$$
  v=q^{(D-d)/2}z
$$
has degree $D$ and satisfies the required identities \eqref{eq:rank1-factor-identities}. 2 such factors
differ by a sign: over the fraction field they have the form $v'=fv$,
and $v'v'^{\mathsf T}=vv^{\mathsf T}$ implies $f^2=1$.

We next prove smooth dependence. In a local frame, the factors lie in
the finite-dimensional space of vector polynomials of degree $D$. The
subset defined by $|v|^2=q^D$ is compact: it is closed, and every
element has $L^2$ norm 1 on the unit sphere. The map
$$
  v\longmapsto vv^{\mathsf T}
$$
has injective derivative at every nonzero $v$. Indeed, if
$v\dot v^{\mathsf T}+\dot v v^{\mathsf T}=0$, choose $v_j\neq0$ over
the fraction field. The $jj$ entry gives $\dot v_j=0$, and the $ij$
entries then give $\dot v_i=0$ for all $i$.

After composing with a suitable linear projection of the target, the
inverse function theorem supplies a smooth inverse near a chosen
factor. Compactness shows that every factor of a sufficiently nearby
matrix lies near one of the 2 factors $v,-v$: otherwise a convergent
subsequence would give a 3rd factor of the limiting matrix. Thus a
local choice of sign produces a smooth factor of the given family.
On overlaps the signs are locally constant and define a flat real line
bundle. Orthogonal changes of frame preserve the construction.

Finally, the metric connection preserves $q$. Differentiating
$q^DP=vv^{\mathsf T}$ along the geodesic flow and using $XP=0$ gives
$v(Xv)^{\mathsf T}+(Xv)v^{\mathsf T}=0$. The preceding injectivity
argument yields $Xv=0$.
\end{proof}

\begin{lemma}[The highest harmonic component]\label{rank1:top-splitting}
Let $v$ be a homogeneous normal vector polynomial of odd degree
$\ell\geq3$, and write its componentwise harmonic decomposition as
$$
  v=\eta_\ell+q\eta_{\ell-2}+q^2\eta_{\ell-4}+\cdots.
$$
There is a unique decomposition
\begin{equation}\label{eq:rank1-top-splitting}
  \eta_\ell=H_\ell+I_\ell(\tau),\qquad
  I_\ell(\tau)=y\tau-
      \frac{q\,\operatorname{grad}_y\tau}{n+2\ell-4},
\end{equation}
where $H_\ell$ is harmonic and normal and $\tau$ is a harmonic scalar
polynomial of degree $\ell-1$. The 2 components depend linearly on
$\eta_\ell$, and these linear maps commute with parallel transport.
\end{lemma}

\begin{proof}
Since $y\cdot v=0$, the polynomial $y\cdot\eta_\ell$ is divisible by
$q$. A contraction with $y$ of a harmonic vector of degree $\ell$ has
only harmonic degrees $\ell+1$ and $\ell-1$. Consequently
$y\cdot\eta_\ell=qT$, where $T$ is harmonic of degree $\ell-1$.
For a harmonic polynomial $\tau$ of this degree,
$$
 \Delta_y(y\tau)=2\operatorname{grad}_y\tau,\qquad
 \Delta_y(q\operatorname{grad}_y\tau)=2(n+2\ell-4)\operatorname{grad}_y\tau.
$$
Thus $I_\ell(\tau)$ is harmonic. Euler's identity gives
$$
 y\cdot I_\ell(\tau)
   =\frac{n+\ell-3}{n+2\ell-4}\,q\tau.
$$
Taking $\tau=(n+2\ell-4)T/(n+\ell-3)$ gives the required normal
remainder and proves uniqueness. The formulas use only the metric,
so the last assertion follows.
\end{proof}

\begin{lemma}\label{rank1:source-separation}
Suppose that the coefficients in Lemma~\ref{rank1:top-splitting} are smooth
on a Riemannian manifold and that $Xv=0$. Then the normal harmonic
degree-$\ell+1$ projection of $\nabla_yH_\ell$ vanishes. If
$H_\ell=0$, then $X_+\tau=0$.
\end{lemma}

\begin{proof}
Put $F=X_+\tau$, a scalar harmonic polynomial of degree $\ell$.
Since multiplication by $y$ and $q$ is parallel, the degree-$\ell+1$
harmonic part of $XI_\ell(\tau)$ is the harmonic projection of $yF$,
namely
\begin{equation}\label{eq:rank1-source-vector}
  yF-\frac{q\,\operatorname{grad}_yF}{n+2\ell-2}.
\end{equation}
All other terms are divisible by $q$ and have no harmonic component of
degree $\ell+1$. The vector in \eqref{eq:rank1-source-vector} is orthogonal on the sphere to
every normal harmonic vector of degree $\ell+1$: the 1st term pairs
to 0 by normality, and $\operatorname{grad}_yF$ has degree $\ell-1$. Projecting
the highest-degree equation in $Xv=0$ onto this normal harmonic space
therefore gives the asserted condition on $H_\ell$.

If $H_\ell=0$, the vector in \eqref{eq:rank1-source-vector} itself vanishes. Its contraction
with $y$ is
$$
 \frac{n+\ell-2}{n+2\ell-2}\,qF.
$$
Hence $F=0$, as claimed.
\end{proof}

\begin{theorem}\label{rank1:exclusion}
Let $(M^n,g)$ be closed, connected and strictly quarter-pinched with
negative sectional curvature, and suppose that $n\geq4$. There is no smooth, even,
flow-invariant rank-1 normal orthogonal projector with finite Fourier
content on $SM$.
\end{theorem}

\begin{proof}
After a constant rescaling of the metric, assume
$-1\leq\Sec_g\leq-\delta$ with $\delta>1/4$. Apply
Lemma~\ref{rank1:dyad-lift} and pass, if necessary, to the 2-fold cover
trivializing its flat line bundle. We obtain a global homogeneous normal
vector polynomial $v$ of odd degree $D$ with
$$
  q^DP=vv^{\mathsf T},\qquad |v|^2=q^D,\qquad Xv=0.
$$

Suppose that its current degree is $\ell\geq3$. Decompose the highest
harmonic component by Lemma~\ref{rank1:top-splitting} as in \eqref{eq:rank1-top-splitting}. By
Lemma~\ref{rank1:source-separation}, the normal harmonic projection of the
highest derivative of $H_\ell$ vanishes. The differential estimate
\eqref{eq:diff-vector-bound} and the positive curvature estimate
\eqref{eq:diff-vector-coercivity} therefore give $H_\ell=0$. These estimates
apply pointwise to $A=-\mathcal R_g$ with
$\delta\leq\Sec_A\leq1$ and $\delta>1/4$.

The same lemma now gives $X_+\tau=0$. Since $\tau$ has positive degree
$\ell-1$, Lemma~\ref{pr:scalar-vanishing} gives $\tau=0$. Thus
$\eta_\ell=0$ and $v=qv'$ for a homogeneous polynomial $v'$ of degree
$\ell-2$. Division by $q$ is a fixed linear operation on its image in
the coefficient space, so $v'$ is smooth. Moreover,
\begin{equation}\label{eq:rank1-reduced-identities}
  y\cdot v'=0,\qquad Xv'=0,\qquad |v'|^2=q^{\ell-2}.
\end{equation}
Repeating this argument reduces the factor to degree 1.

Write the resulting factor as $v_1(y)=Jy$. Normality and its norm
identity in \eqref{eq:rank1-reduced-identities} imply $J^{\mathsf T}=-J$ and $J^{\mathsf T}J=\Id$. The
original projector is therefore the smooth quadratic projector
\begin{equation}\label{eq:rank1-quadratic-projector}
  P(y)=\frac{(Jy)(Jy)^{\mathsf T}}{q(y)}.
\end{equation}
Corollary~\ref{pr:quadratic-exclusion} now excludes the projector in \eqref{eq:rank1-quadratic-projector}. This
contradiction finishes the proof.
\end{proof}

\section{Spinorial description of the exceptional invariant structures}
\label{spin:section}

We now consider the odd flow-invariant orthogonal almost complex
structure on the normal bundle in dimension 7 and the odd
flow-invariant $G_2$-structure in dimension
8 furnished by Proposition~\ref{main:alternatives}. We first express these
structures as polynomial spinor projectors. After passing to a finite
cover, their polynomial factors are smooth, including at points where
the Fourier degree decreases. The highest nonzero coefficient then
satisfies a generalized-gradient equation to which we will apply the
curvature estimates below. No a priori degree estimate for the original
differential form is needed.

We use the real spin representations and the Clifford convention
of \cite{LawsonMichelsohn89}:
\begin{equation}\label{eq:spin-lift-clifford-convention}
c(a)c(b)+c(b)c(a)=-2\langle a,b\rangle\Id,
\qquad c(a)^*=-c(a).
\end{equation}
For an orthonormal basis $(e_i)$, set $c_i=c(e_i)$. Our convention for
Clifford multiplication by a $p$-form is
$$
c(\omega)=\sum_{i_1<\cdots<i_p}
\omega_{i_1\cdots i_p}c_{i_1}\cdots c_{i_p}.
$$
In dimension 7, $\Sigma$ is a real irreducible spin representation
of dimension 8. In dimension 8, $\Sigma^+$ and $\Sigma^-$ are
the real half-spin representations and $\Sigma=\Sigma^+\oplus\Sigma^-$.
Clifford multiplication by a vector interchanges the 2 summands,
and its adjoint maps back to the original chirality. Once a spin
structure is chosen, the same symbols denote the associated spinor
bundles. Polynomial functions on a unit sphere are
viewed in $\R[y]/(q-1)$, where $q(y)=|y|^2$, with the filtration by
total degree. This degree is the Fourier degree from
Section~\ref{main:reduction}.

\subsection{Polynomial spinor projectors}

\begin{lemma}[Polynomial rank-1 projectors]
\label{spin:affine}
For $n=7,8$, the ring $A=\C[y_1,\ldots,y_n]/(q-1)$ is factorial,
its units are the nonzero constants, and degrees add under multiplication
of nonzero elements. Every real polynomial orthogonal projector of rank
1 on $S^{n-1}$ has the form
\begin{equation}\label{eq:spin-lift-rank-one-factorization}
P=ss^{\mathsf T},\qquad s^{\mathsf T}s=1,
\end{equation}
where $s$ is real and polynomial. The factor is unique up to sign, and
$\deg P=2\deg s$.

A smooth family of these projectors with uniformly bounded degree has
smooth local polynomial factors of uniformly bounded degree. The
transition functions are locally constant signs on the parameter space.
\end{lemma}
\begin{proof}
Put $x=y_1+iy_2$ and $u=y_1-iy_2$. Then
$$
\begin{aligned}
A[x^{-1}]&=\C[x,x^{-1},y_3,\ldots,y_n],\\
A/(x)&=\C[u,y_3,\ldots,y_n]/(y_3^2+\cdots+y_n^2-1).
\end{aligned}
$$
The 2nd ring is a domain: a factorization of its defining quadratic
would express its homogenization, of rank $n-1\geq6$, as a product of
2 linear forms, whose rank is at most 2. Thus $x$ is prime, and
Nagata's criterion \cite[Lemma~2]{Nagata57} proves that $A$ is factorial.
In particular its rank-1 projective modules are free. The associated
graded ring is $\C[y]/(q)$, which is a domain. Hence degrees add and a
unit has degree 0. The real sphere is Zariski dense in the affine
quadric; for example, division by $q-1$ in 1 variable and evaluation
on the 2 sheets over an open real ball proves this directly.

The same algebra also applies to a Hermitian rank-1 idempotent, with
conjugation fixing the variables. Choose a free generator $w$ of its
image and write $P=wa$, where $a$ is a polynomial row and $aw=1$.
Hermitian symmetry, with $h=w^*w$, gives
$$
ha=w^*,\qquad h\,aa^*=1.
$$
Thus $h$ is a positive constant, and $P=ww^*/h$. If $P$ is real,
$\overline w$ and $w$ differ by a constant of modulus 1. A constant
phase change therefore makes $w$ real, and normalization gives $s$ as in \eqref{eq:spin-lift-rank-one-factorization}.
2 real unit factors differ by sign. A diagonal entry corresponding
to a component of maximal degree gives $\deg P=2\deg s$.

For the family assertion, the degree identity places all factors in a
fixed finite-dimensional coefficient space. Its real normalized factors
form a compact set because their $L^2$ norm on the unit sphere is 1.
The derivative of
$s\mapsto ss^{\mathsf T}$ is injective at every such factor: the identity
$$
\dot s\,s^{\mathsf T}+s\,\dot s^{\mathsf T}=0
$$
first gives $\langle s,\dot s\rangle=0$ by taking the trace and then
$\dot s=0$ by application to $s$. The compactness and local inverse
argument in Lemma~\ref{rank1:dyad-lift} now applies in these coefficient
spaces. It gives the 2 smooth local branches even when the degree
decreases, with locally constant transition signs.
\end{proof}

\begin{lemma}[Spinor factors in dimension 8]
\label{spin:pair8}
Suppose that a closed oriented 8-manifold carries an odd
flow-invariant normal $G_2$-structure with finite Fourier content, compatible
with the metric. On a finite base cover there are smooth polynomial
unit spinors $s,t$ of opposite chirality such that
\begin{equation}\label{eq:spin-lift-pair8}
t(v)=c(v)s(v),\qquad Xs=Xt=0.
\end{equation}
The spinors have opposite antipodal parity and uniformly bounded degree.
\end{lemma}
\begin{proof}
Let $\varphi$ be the normal $G_2$ 3-form. Since $G_2$ is simply
connected, its defining orthogonal representation lifts to $\Spin(7)$,
so the 2nd Stiefel--Whitney class $w_2(N)$ vanishes. The decomposition
$\pi^*TM=\R v\oplus N$ and the
isomorphism on $H^2(-;\Z/2)$ induced by the $S^7$-bundle $SM\to M$
give $w_2(TM)=0$. Thus $M$ admits a spin structure
\cite[Chapter~II, Theorem~1.7]{LawsonMichelsohn89}; we fix one.

By the standard spinorial correspondence, metric-compatible Cayley
forms of either orientation are in bijection with unit real spinors in
the corresponding half-spin representation, modulo sign
\cite[Section~2.2]{MartinMerchan20}. We record the formula for their
rank-1 projectors with our normalization.
For a unit real half-spinor $s$, put
$$
(\Phi_s)_{ijkl}=-\langle c_i c_j c_k c_l s,s\rangle
\qquad(i<j<k<l),
$$
and extend the components by antisymmetry.
On the half-spin representation for which these 4-forms have the
corresponding Hodge sign, the Clifford trace identities give
\begin{equation}\label{eq:spin-lift-cayley-projector}
\tr\bigl(c(\Phi)c(\Psi)\bigr)
   =16\langle\Phi,\Psi\rangle,
\qquad
ss^{\mathsf T}=\frac18\Id-\frac1{16}c(\Phi_s).
\end{equation}
Indeed, Clifford multiplication identifies that $35$-dimensional space
of 4-forms with the trace-free symmetric endomorphisms of the
8-dimensional half-spin representation. Pairing the proposed formula
in \eqref{eq:spin-lift-cayley-projector} with $c(\Psi)$ gives
$\langle c(\Psi)s,s\rangle=-\langle\Phi_s,\Psi\rangle$;
the trace gives the scalar term $1/8$. This proves the coefficient
$1/16$ and fixes the sign with our Clifford convention \eqref{eq:spin-lift-clifford-convention}.

Choose the sign of $\varphi$ so that its induced orientation agrees
with the normal orientation for which $(v,N)$ has the orientation of
$M$. This choice is consistent on each connected component, since
both orientations reverse under $v\mapsto-v$. The self-dual extension
of $\varphi$ is the 1st form below
\cite[Remark~7.3(4)]{MartinMerchan20}; its paired form has the
opposite Hodge sign, where $v^\flat=g(v,\cdot)$ and $*$ is the Hodge
star of the oriented 8-dimensional tangent space:
\begin{equation}\label{eq:spin-lift-cayley-pair}
\Phi_+(v)=v^\flat\wedge\varphi(v)
                  +*(v^\flat\wedge\varphi(v)),\qquad
\Phi_-(v)=-v^\flat\wedge\varphi(v)
                  +*(v^\flat\wedge\varphi(v)).
\end{equation}
The spinor projectors of the forms in \eqref{eq:spin-lift-cayley-pair} are real, even, flow-invariant, and polynomial
of uniformly bounded degree. The signs in this pair can also be checked
by Clifford multiplication. If $R_v=2v\otimes v-\Id$, then
$$
\Phi_{c(v)s}=R_v\cdot\Phi_s.
$$
The reflection $R_v$ fixes $v$ and negates $v^\perp$; it therefore
negates $v^\flat\wedge\varphi$ and fixes its Hodge dual. Thus $c(v)$
maps the 1st spinor line to the 2nd one.

By Lemma~\ref{spin:affine}, the 1st projector has smooth local polynomial
unit factors of uniformly bounded degree. Their transition signs define
a flat real line bundle on $M$, trivial on an at-most-2-fold cover.
On this cover choose $s$ globally and put $t=c(v)s$. Injectivity of
the differential of the tensor-square map gives $Xs=0$ from invariance
of the 1st projector, and parallelness of Clifford multiplication
along geodesics gives $Xt=0$, as in \eqref{eq:spin-lift-pair8}. Finally $s(-v)=\pm s(v)$ because these
are unit factors of the same projector; the sign is constant on each
connected component. Multiplication by $c(v)$ reverses it.
\end{proof}

\begin{lemma}[A spin structure after a finite cover]
\label{spin:spin7}
An odd orthogonal almost complex structure $J$ on the normal bundle of
$SM$ in dimension 7 makes an at-most-2-fold base cover spin.
We also denote by $J$ the 2-form $(u,w)\mapsto g(Ju,w)$ on $N$,
extended by 0 on the geodesic direction. On that cover one can choose
the Clifford volume sign so that
\begin{equation}\label{eq:spin-lift-rank-two-projector}
P(v)=\frac14\bigl(\Id-c(v)c(J(v))\bigr)
\end{equation}
is a real orthogonal spinor projector of rank 2 commuting with $c(v)$.
If $J$ is flow-invariant and has finite Fourier content, then $P$ is
flow-invariant, even, and of finite degree.
\end{lemma}
\begin{proof}
Regard $(N,J)$ as a complex rank-3 bundle $E$. The antipodal map
$a$ satisfies $a^*E\simeq\overline E$. The $S^6$-bundle induces
isomorphisms from the base in integral $H^2$ and in mod-2 $H^1,H^2$.
Since $a$ covers the identity, it acts trivially on integral $H^2$,
whereas conjugation changes the sign of the 1st Chern class $c_1(E)$.
Thus $2c_1(E)=0$. The integral Bockstein sequence shows that $c_1(E)$
is the Bockstein image of a class $\eta\in H^1(SM;\Z/2)$ pulled back
from the base. The at-most-2-fold cover classified by this base class
kills $\eta$ and hence $c_1(E)$. Now
$$
\pi^*w_2(TM)=w_2(N)=c_1(E)\pmod2
$$
by \cite[Chapter~II, Remark~1.8]{LawsonMichelsohn89}. Injectivity in
mod-2 $H^2$ and the spin obstruction
\cite[Chapter~II, Theorem~1.7]{LawsonMichelsohn89} show that the cover
is spin.

Choose the sign of $J$ to agree with the normal orientation and use
$c_1\cdots c_7=\Id$. At $v=e_7$, take a compatible unitary frame with
dual coframe $(e^i)$, so that
$J=e^1\wedge e^2+e^3\wedge e^4+e^5\wedge e^6$. The commuting operators
$K_j=c_{2j-1}c_{2j}$ square to $-\Id$, and the volume relation gives
$$
c(v)c(J)=K_1K_2+K_1K_3+K_2K_3.
$$
Its eigenvalues are $-3$ and $1$. Its trace is 0, so their real
multiplicities are 2 and 6. This proves the projector assertion \eqref{eq:spin-lift-rank-two-projector}.
Since $J$ is normal, $c(v)$ commutes with $c(J)$; the remaining claims
follow from oddness and from parallelness of Clifford multiplication.
\end{proof}

\subsection{Smooth factors and the highest coefficient}

The rank-2 projector requires a separate smoothness argument: its
factor map need not be an immersion when the degree decreases. We use
elementary subintegral extensions and the smoothness theorem of
Joris--Preissmann \cite{JP87}. Recall that an integral extension is
subintegral if it induces a bijection on prime spectra and
isomorphisms on residue fields \cite{Swan80}.

\begin{lemma}[Elementary subintegral extensions]
\label{spin:subintegral}
A finite subintegral extension $A\subset B$ of Noetherian rings admits
a finite chain
\begin{equation}\label{eq:spin-lift-subintegral-chain}
A=A_0\subset A_1\subset\cdots\subset A_r=B,
\qquad A_{j+1}=A_j[b_j],\qquad b_j^2,b_j^3\in A_j.
\end{equation}
\end{lemma}
\begin{proof}
Suppose that $A\ne B$, and choose a prime $\mathfrak p$
minimal in the support of $B/A$. The conductor
$$
I=(A_{\mathfrak p}:B_{\mathfrak p})
  =\operatorname{Ann}_{A_{\mathfrak p}}
       (B_{\mathfrak p}/A_{\mathfrak p})
$$
contains a power of the maximal ideal, since the localized quotient
has finite length. Hence $C=A_{\mathfrak p}/I\subset
D=B_{\mathfrak p}/I$ is a proper Artinian local subintegral extension.
Its residue fields coincide. If $\mathfrak n$ is the maximal ideal of
$D$, then $\mathfrak n\not\subset C$, for otherwise $D=C$. Choose the
largest $k\geq1$ such that $\mathfrak n^k\not\subset C$ and take
$\overline b\in\mathfrak n^k\setminus C$. Its square and cube belong
to $C$. A lift gives $b\in B_{\mathfrak p}\setminus A_{\mathfrak p}$
with $b^2,b^3\in A_{\mathfrak p}$, because $I$ is also an ideal of
$B_{\mathfrak p}$ contained in $A_{\mathfrak p}$.

Clear denominators to obtain $b_0\in B\setminus A_{\mathfrak p}$,
and then choose $u\in A\setminus\mathfrak p$ such that
$ub_0^2,ub_0^3\in A$. Thus $b_1=ub_0\notin A$ and
$$
b_1^2=u(ub_0^2)\in A,\qquad
b_1^3=u^2(ub_0^3)\in A.
$$
This gives a proper elementary intermediate extension. Repeat with
each intermediate ring. Such rings are finite $A$-modules, and their
inclusions in $B$ are still finite and subintegral. The ascending chain
condition in the finite $A$-module $B$ forces termination, necessarily
at $B$.
\end{proof}

\begin{theorem}[Polynomial spinor factors in dimension 7]
\label{spin:lift7}
Let $P_x(y)$ be a smooth family of even polynomial orthogonal projectors
of rank 2 on $S^6$, with uniformly bounded degree, satisfying
$[P_x(y),c(y)]=0$. Locally in $x$ there are smooth odd polynomial unit
spinors $z_x$, of uniformly bounded degree, such that
\begin{equation}\label{eq:spin-lift-rank-two-factorization}
P_x=z_xz_x^{\mathsf T}+h_xh_x^{\mathsf T},\qquad h_x=c(y)z_x.
\end{equation}
These factors are unique up to sign. For an invariant family over $SM$
there is a smooth real function $\lambda$ on $M$ satisfying
\begin{equation}\label{eq:spin-lift-factor-evolution}
Xz=\lambda h,\qquad Xh=-\lambda z.
\end{equation}
After a further at-most-2-fold base cover, the factors and these
equations are globally defined.
\end{theorem}
\begin{proof}
First work over $A=\C[y_1,\ldots,y_7]/(q-1)$. The Hermitian idempotent
$$
P_+=\frac12(P-ic(y)P)
$$
has rank 1. The proof of Lemma~\ref{spin:affine} provides a polynomial
generator $w$ of constant positive Hermitian norm. As
$P_+(-y)=\overline{P_+(y)}$, a constant phase change makes
$w(-y)=\overline{w(y)}$. Write $w=a+ib$, with $a$ even and $b$ odd.
The identity $c(y)w=iw$ gives $a=c(y)b$. Moreover skewness gives
$w^{\mathsf T}w=0$, and hence
$$
|a|^2=|b|^2=\tfrac12w^*w,\qquad \langle a,b\rangle=0.
$$
Normalization produces $z$ and $h=c(y)z$, giving \eqref{eq:spin-lift-rank-two-factorization}. The degree identity
$\deg P_+=2\deg w$ gives a uniform bound on their degrees. 2 such
normalized odd factors differ by sign: the corresponding generators
of $P_+$ differ by a constant unit, and antipodality and normalization
make that constant real with square 1.

We next prove smooth dependence. Let $L_D$ be the finite-dimensional
space of odd spinor polynomials on the sphere of degree at most a fixed
sufficiently large $D$. In coefficient coordinates consider
$$
\Psi(z)=zz^{\mathsf T}+(c(y)z)(c(y)z)^{\mathsf T},\qquad
X_D=\{z\in L_D:z^{\mathsf T}z=1\},
$$
where $X_D$ has the reduced algebraic structure. Complexify these
coefficient equations, using transpose throughout. The homogeneous
quadratic map $\Psi$ has fibre $\{0\}$ over $0$. Indeed, if $z\ne0$ and
$\Psi(z)=0$, the 2 rank-1 matrices imply, over $\Frac(A)$,
that $c(y)z=az$ and $a^2=-1$. Thus $a=\pm i$, contradicting the
opposite parities of $z$ and $c(y)z$.

If $t_i$ are the coefficient variables and $f_j$ the coordinate
quadratics of $\Psi$, the homogeneous Nullstellensatz gives
$(t_1,\ldots,t_s)^r\subset(f_1,\ldots,f_k)$ for some $r$.
Induction on homogeneous degree then shows that $\C[t]$ is a finite
$\C[f]$-module: monomials of degree at least $r$ are sums $f_jg_j$,
where each $g_j$ has degree 2 less. Thus $\Psi$ is finite onto its
affine image. Finiteness descends to $\R$ and remains valid after
restriction to $X_D$.

On $X_D$, each complex fibre is of the form $\{z,-z\}$. To see
this, put $h=c(y)z$, $w=h+iz$, and $w_-=h-iz$. Then
$$
w_-^{\mathsf T}w=2,\qquad P_+=\tfrac12ww_-^{\mathsf T}.
$$
Equal projectors give $w'=aw$ with $a\in A^*=\C^*$. Since
$w(-y)=w_-(y)$, the same constant multiplies $w_-$; normalization
therefore gives $a^2=1$. This argument holds over every algebraically
closed extension of characteristic 0. If $Z_D$ is the reduced affine
image, the induced inclusion
$$
A_0=\R[Z_D]\subset B_0=\R[X_D]^{\{\pm1\}}
$$
is finite and subintegral. In fact, after quotienting by sign its
geometric fibres are singletons; taking invariants commutes with field
extensions because averaging over $\{\pm1\}$ is a projection. Thus
the map is bijective on spectra and its
residue field extensions are purely inseparable, hence trivial in
characteristic 0. No reducedness assertion about the fibres is
needed.

The set $X_D(\R)$ is compact because every element has sphere $L^2$
norm 1. Thus $X_D(\R)/\{\pm1\}$ maps homeomorphically onto
$\Psi(X_D(\R))$, and the given smooth family has a continuous lift
to this quotient. Apply Lemma~\ref{spin:subintegral} to $A_0\subset B_0$.
The original smooth family pulls back every coordinate of $A_0$
smoothly. At each step in \eqref{eq:spin-lift-subintegral-chain}, the next generator has smooth square and cube,
so it is smooth by \cite[p.~196, Theorem~A]{JP87}. Induction gives
smooth pullbacks of all invariant coordinates.

Choose a nonzero real coefficient $t_j$ of one factor. Its invariant
square remains positive nearby; select its smooth signed square root
and recover every other coefficient from
$$
t_i=\frac{t_it_j}{t_j}.
$$
The numerator is invariant, so these are smooth local polynomial
factors, including wherever their degree decreases. Their transition
signs define a flat real line bundle on the base, trivial on an
at-most-2-fold cover.

Finally suppose $XP=0$. Differentiate $Pz=z$, $|z|^2=1$, and
$h=c(y)z$ to obtain
$$
Xz=\alpha h,\qquad Xh=-\alpha z,
\qquad \alpha=\langle Xz,h\rangle.
$$
Here $\alpha$ is smooth, polynomial in $y$, and unchanged by the
simultaneous sign of the factors. Their actual fibre degrees differ by
1: Clifford multiplication changes degree by at most 1 in either
direction, and the factors have opposite parity. On an open set where
both degrees are locally constant, differentiate the factor of smaller
degree. Its higher coefficients vanish on that open set, so their
derivatives vanish there as well. Its derivative therefore has degree
at most the degree of the larger
factor. Additivity of degree applied to $\alpha$ times a component of
maximal degree of the larger factor forces $\deg\alpha=0$, unless
$\alpha=0$.

These open sets form a dense subset of the base, since only finitely
many coefficient degrees occur. The smooth, sign-independent function
$$
\lambda(x)=\int_{S_xM}\alpha(x,v)\,d\sigma(v)
$$
agrees with $\alpha$ there. Continuity of each coefficient extends the
identity to the whole base. This proves \eqref{eq:spin-lift-factor-evolution} without
assuming that $\lambda$ is constant.
\end{proof}

Let $D_y=\sum_i c_i\partial_{y_i}$ be the Dirac operator in the fibre
variable; a spinor polynomial $m$ is monogenic if $D_ym=0$.
The Cartan generalized gradient is the orthogonal projection of the
covariant derivative onto the Cartan product with the vector
representation, whose highest weight is the sum of the 2 highest
weights; see \cite[Sections~1--2]{CGH00}.

\begin{lemma}[The highest spinor coefficient]
\label{spin:top}
For either pair in Lemma~\ref{spin:pair8} and Theorem~\ref{spin:lift7}, order the factors as
$F=c(v)G$, with global degrees $\ell+1$ and $\ell$. Then
\begin{equation}\label{eq:spin-lift-top-coefficient}
F_{\ell+1}=c(v)m_\ell,
\end{equation}
where $m_\ell$ is a smooth, nonzero section of the bundle of homogeneous
monogenic spinor polynomials of degree $\ell$. The Cartan generalized
gradient of $m_\ell$, for the induced spin connection, vanishes. The
section $m_\ell$ may vanish at individual base points.
\end{lemma}
\begin{proof}
The Clifford inverse relation and opposite parities show that the global
degrees differ by 1. In dimension 7 take $(F,G)=(h,z)$ or
$(z,-h)$ according to which degree is larger. Then
\begin{equation}\label{eq:spin-lift-ordered-evolution}
XG=\lambda F,\qquad XF=-\lambda G.
\end{equation}
The same equations hold in dimension 8 with $\lambda=0$. We use
global degrees throughout, so this choice remains fixed at zeros of the
highest coefficient.

The Fischer decomposition for
harmonic spinor polynomials \cite[Section~2]{Lavicka11} gives
\begin{equation}\label{eq:spin-lift-fischer-top}
G_\ell=m_\ell+c(y)n_{\ell-1},\qquad
D_ym_\ell=D_yn_{\ell-1}=0,
\end{equation}
with the 2nd summand absent for $\ell=0$. Multiplication by $c(y)$
sends the 2nd summand in \eqref{eq:spin-lift-fischer-top} to $-qn_{\ell-1}$, whose spherical degree is
$\ell-1$. Lower harmonics of $G$ also contribute only degrees at most
$\ell-1$. Hence $F_{\ell+1}=c(y)m_\ell$, as in \eqref{eq:spin-lift-top-coefficient}, and this coefficient is
not identically 0.

The connection preserves monogenic polynomials. Writing $m=m_\ell$,
we have
$$
\Delta_y(Xm)=2\sum_i\partial_{y_i}\nabla_i m,
\qquad
D_y\Bigl(\sum_i\partial_{y_i}\nabla_i m\Bigr)=0.
$$
For $\ell\geq1$, the Fischer decomposition is consequently
\begin{equation}\label{eq:spin-lift-fischer-derivative}
Xm=U_{\ell+1}+c(y)V_\ell+qW_{\ell-1},\qquad
W_{\ell-1}=\frac{\sum_i\partial_{y_i}\nabla_i m}{n+2\ell-2},
\end{equation}
where $U,V,W$ are monogenic; omit $W$ when $\ell=0$. This follows
by first subtracting $qW$ to make the remainder harmonic and then
applying the harmonic-spinor decomposition. Parallelness of Clifford
multiplication and \eqref{eq:spin-lift-fischer-derivative} give
\begin{equation}\label{eq:spin-lift-highest-derivative}
X(c(y)m)=c(y)U_{\ell+1}-qV_\ell+qc(y)W_{\ell-1}.
\end{equation}
The only harmonic term of degree $\ell+2$ in \eqref{eq:spin-lift-highest-derivative} is $c(y)U_{\ell+1}$.
No lower harmonic of $F$ contributes that degree, while
$XF=-\lambda G$ in \eqref{eq:spin-lift-ordered-evolution} has degree at most $\ell$. Thus $U_{\ell+1}=0$.

The component $U_{\ell+1}$ is the image of this Cartan component under
the symbol followed by monogenic projection.
Indeed, the monogenic target of degree $\ell+1$ is the highest-weight
summand and has multiplicity 1. Its symbol is nonzero: for a complex
null covector $a$ and a spinor $\xi\ne0$ with $c(a)\xi=0$, it sends
$a\otimes(a\cdot y)^\ell\xi$ to
$(a\cdot y)^{\ell+1}\xi$. Therefore its vanishing is exactly the
claimed generalized-gradient equation. All projections are parallel
linear fibre maps, so the argument is valid also at zeros of $m_\ell$.
\end{proof}

\subsection{Curvature estimates for monogenic spinor polynomials}

Let $\mathcal M_\ell$ denote the space of homogeneous spinor
polynomials of degree $\ell$ annihilated by the vertical Dirac
operator $D_y$; in dimension 8, either chirality is considered
separately. The same notation denotes the associated bundle on a spin
manifold. We now estimate the curvature endomorphisms on these
spaces. The differential identities are the universal
Weitzenb\"ock formula and the twisted Schr\"odinger--Lichnerowicz
formula; see \cite{CGH00,SemmelmannWeingart10,LawsonMichelsohn89}.
Throughout this subsection, $n=7$ or $8$. The polynomial spaces carry
the normalized sphere $L^2$ metric; pairings inside sphere integrals
use the pointwise spinor inner product. Norms in algebraic identities
are fibre norms, and those in integrated differential identities use
Riemannian volume on the base and this fibre metric.

Monogenic polynomials are harmonic because $D_y^2=-\Delta_y$.
If $m\in\mathcal M_\ell$, then $F(y)=c(y)m(y)$ is harmonic of degree
$\ell+1$, since $\Delta_yF=2D_ym+c(y)\Delta_ym=0$.
Also $\|F\|=\|m\|$ on the unit sphere. Use the scalar abbreviations
$$
E=n+2\ell-2,\qquad N_\ell=n+2\ell,\qquad
D=n+\ell-1,\qquad L=(\ell-1)(\ell+n).
$$
For an algebraic curvature tensor $A$, let $Q=Q_{\ell,A}$ be the
curvature endomorphism of the representation $\mathcal M_\ell$:
$$
Q=-\sum_{\alpha,\beta}A_{\alpha\beta}
                         \rho(\omega_\alpha)\rho(\omega_\beta),
$$
Here $(\omega_\alpha)$ is an orthonormal basis of $\Lambda^2\R^n$,
$A_{\alpha\beta}=\langle A\omega_\alpha,\omega_\beta\rangle$,
and $\rho$ is the infinitesimal $\Spin(n)$ action on $\mathcal M_\ell$,
normalized by
$$
\rho(e_i\wedge e_j)
=-y_i\partial_j+y_j\partial_i+\tfrac12 c_ic_j\qquad(i<j).
$$
Let $K$ be the compression to $\mathcal M_\ell$ of the curvature
endomorphism acting on the polynomial indices alone. Define $K^+$ by
pulling back the quadratic form of the curvature endomorphism acting
on polynomial indices in degree $\ell+1$ along the isometric embedding
$m\mapsto c(y)m$. Its
superscript refers to the polynomial degree, independently of
half-spin chirality. We also write
$$
S=\scal_A\Id,\qquad C=Q-K-S/8,\qquad
R_\ell=\sum_{j=1}^{\ell}\Ric_A^{(j)}.
$$
In the last expression the Ricci endomorphism acts on each symmetric
coefficient index, followed by orthogonal projection to
$\mathcal M_\ell$. Thus $C$ is the mixed spinor--tensor curvature
term, since the spinor curvature endomorphism is $S/8$.

With the Jacobi operator $A_v$ defined above, we have
$$
\langle Km,m\rangle
 =\int_{S^{n-1}}(A_v)_{ab}
                 \langle\partial_am,\partial_bm\rangle\,d\sigma(v).
$$
Consequently, if $\Sec_A\geq\delta$, then
\begin{equation}\label{eq:spin-est-polynomial-bounds}
K\geq\delta\ell(\ell+n-2)\Id,\qquad
K^+\geq\delta(\ell+1)(\ell+n-1)\Id.
\end{equation}
We use real inner products; for complexified representations the same
formulas hold with the real part of the Hermitian inner product.

\begin{lemma}\label{spin:negative}
Let $W\subset\R^n$ be a $d$-dimensional Euclidean subspace, let
$P:W\longrightarrow\Sigma$ be a spinor-valued 1-form, and let
$\Gamma P=\sum_a c(e_a)P(e_a)$, where $(e_a)_{a=1}^d$ is an
orthonormal basis of $W$. For a real parameter $c\geq0$, put
$$
e=|P|^2,\qquad\gamma=\Gamma P,\qquad
T_c=cP^*P-e\Id+P^*\Gamma^*\gamma+(\Gamma^*\gamma)^*P.
$$
Then
$$
\tr(T_c)_-\leq de-|\gamma|^2.
$$
In particular, if $J$ is a self-adjoint endomorphism of $W$ satisfying
$\delta\Id\leq J\leq\Id$, with $0\leq\delta\leq1$,
then
$$
\tr(JT_c)\geq(\delta c-d)e+(1+\delta)|\gamma|^2.
$$
\end{lemma}
\begin{proof}
Let $\Pi$ be an orthogonal projection of rank $r$. Decompose
$P=P_1+P_2$ according to $W=\operatorname{im}\Pi\oplus\ker\Pi$, and
write $e_i=|P_i|^2$, $a=\Gamma P_1$, and $b=\Gamma P_2$. Then
$$
de-|\gamma|^2+\tr(\Pi T_c)
 =(d-r+c)e_1+(d-r)e_2+|a|^2-|b|^2\geq0,
$$
because $|b|^2\leq(d-r)e_2$. Taking $\Pi$ to be the negative spectral
projection proves the 1st assertion. Since
$$
\tr T_c=(c-d)e+2|\gamma|^2,
$$
the estimate
$$
\tr(JT_c)\geq\delta\tr T_c-(1-\delta)\tr(T_c)_-
$$
proves the 2nd.
\end{proof}

\begin{lemma}\label{spin:compression}
For $m\in\mathcal M_\ell$, put $F(v)=c(v)m(v)$ and
$P=d_SF:T_vS^{n-1}\longrightarrow\Sigma$. Define $I_c$ and
$J_{\Ric}$ by
$$
\begin{aligned}
\langle I_cm,m\rangle&=\int_{S^{n-1}}\tr(A_vT_c)\,d\sigma(v),\\
\langle J_{\Ric}m,m\rangle
 &=\int_{S^{n-1}}\Ric_A(v,v)|m(v)|^2\,d\sigma(v).
\end{aligned}
$$
Then
\begin{equation}\label{eq:spin-est-compression}
N_\ell K^+=EK+6Q+S/4,
\end{equation}
\begin{equation}\label{eq:spin-est-compression-ricci}
I_c=(c+D)K^+-DK-LJ_{\Ric}-S,\qquad
N_\ell J_{\Ric}=S+2R_\ell.
\end{equation}
If $\delta\leq\Sec_A\leq1$, then
\begin{equation}\label{eq:spin-est-compression-lower}
I_c\geq D\bigl[(\delta c-n+1)(\ell+1)+(1+\delta)D\bigr]\Id.
\end{equation}
In particular, if $1/4<\Sec_A\leq1$,
\begin{equation}\label{eq:spin-est-compression-quarter}
I_c\geq\frac D4
       \bigl[c(\ell+1)-(4n-9)\ell+n-1\bigr]\Id.
\end{equation}
\end{lemma}
\begin{proof}
All integrals in this proof are over $S^{n-1}$. Write
$A_{abcd}=\langle A(e_a\wedge e_b),e_c\wedge e_d\rangle$, with
repeated frame indices summed from $1$ to $n$, and set
$$
\begin{aligned}
\mathcal C&=-\frac12\int A_{abcd}v_a
                     \langle c_cc_dm,\partial_bm\rangle,\\
\mathcal C_{\mathrm{rad}}
 &=-\int A_{abcd}v_av_cv_p
                     \langle c_pc_dm,\partial_bm\rangle,\\
Z&=\int(\Ric_A)_{bc}v_c\langle m,\partial_bm\rangle.
\end{aligned}
$$
The infinitesimal action on polynomials is
$-v_a\partial_b+v_b\partial_a$. Expanding the tensor-product
curvature and differentiating $F=c(v)m$ therefore give
$$
\langle Qm,m\rangle
 =\langle Km,m\rangle+\mathcal C+\frac{\scal_A}{8}\|m\|^2,
$$
$$
\langle K^+m,m\rangle
 =\langle Km,m\rangle+\langle J_{\Ric}m,m\rangle
                                  +2\mathcal C_{\mathrm{rad}}.
$$
For a homogeneous polynomial $f$ of degree $r$, integration by parts
on the sphere gives
$$
\int v_pf=\frac1{n+r-1}\int\partial_pf.
$$
Apply this identity to the factor $v_p$ in
$\mathcal C_{\mathrm{rad}}$. The remaining factor has degree
$2\ell+1$. Differentiating $v_a$ and $v_c$ gives $\mathcal C-Z$ and
$2\mathcal C$, respectively. Here the 1st contraction uses
$$
\sum_{p,d}A_{pbcd}c_pc_d
 =\frac12\sum_{p,d}A_{cbpd}c_pc_d+(\Ric_A)_{bc}\Id,
$$
which follows from the 1st Bianchi identity and the Clifford
relations. Differentiating the 1st copy of $m$ gives
$2\langle Km,m\rangle$, since
$\sum_pc_pc_d\partial_pm=-2\partial_dm$; differentiating the 2nd
copy gives 0, since $D_y\partial_bm=0$. Hence
$$
N_\ell\mathcal C_{\mathrm{rad}}
 =2\langle Km,m\rangle+3\mathcal C-Z.
$$
A further integration by parts yields
$$
2Z=N_\ell\langle J_{\Ric}m,m\rangle-\scal_A\|m\|^2.
$$
Substituting these identities in the expression for $K^+$ gives
$$
N_\ell\langle K^+m,m\rangle
 =(N_\ell+4)\langle Km,m\rangle+6\mathcal C+\scal_A\|m\|^2,
$$
which proves \eqref{eq:spin-est-compression}.

We next check the Ricci normalization. If $m(v)=T(v,\ldots,v)$,
where $T$ is a trace-free symmetric coefficient tensor, then
$$
\int|m|^2
 =\frac{\ell!|T|^2}{n(n+2)\cdots(n+2\ell-2)}.
$$
The same moment formula applied after differentiation gives
$$
\int(\Ric_A)_{ij}\langle\partial_im,\partial_jm\rangle
 =E\langle R_\ell m,m\rangle.
$$
On the other hand, sphere integration by parts gives the left-hand
side as $EZ$: the 2nd derivatives of $m$ have harmonic degree
$\ell-2$ and are orthogonal to $m$. Thus
$Z=\langle R_\ell m,m\rangle$, proving
$N_\ell J_{\Ric}=S+2R_\ell$.

Finally, monogenicity implies
$$
\Gamma P=-Dm,\qquad (\Gamma^*\Gamma P)(a)=Dc(a)m
                 \quad(a\perp v).
$$
The integrated Clifford cross term in $I_c$ is therefore
$$
2D\bigl(\langle J_{\Ric}m,m\rangle+\mathcal C_{\mathrm{rad}}\bigr)
 =D\langle(K^+-K+J_{\Ric})m,m\rangle.
$$
Since $F$ is a harmonic polynomial of degree $\ell+1$,
$$
|P|^2=(\ell+1)D|F|^2+\tfrac12\Delta_S|F|^2.
$$
Here $\Delta_S$ has nonpositive spectrum, and
$$
\Delta_S\Ric_A(v,v)=-2n\Ric_A(v,v)+2\scal_A.
$$
Integration by parts therefore gives
$$
\int\Ric_A(v,v)|P|^2
 =[(\ell+1)D-n]\langle J_{\Ric}m,m\rangle
                                      +\scal_A\|m\|^2.
$$
Combining this expression with the preceding cross term and the
term $c\langle K^+m,m\rangle$ proves the identity for $I_c$ in \eqref{eq:spin-est-compression-ricci}.
Moreover,
$$
\int|P|^2=(\ell+1)D\|m\|^2,\qquad
\int|\Gamma P|^2=D^2\|m\|^2.
$$
Polarization gives the stated operator identities. The lower bound \eqref{eq:spin-est-compression-lower}
now follows from Lemma~\ref{spin:negative} with $d=n-1$.
\end{proof}

\begin{lemma}\label{spin:gradients}
Let $m$ be a smooth section of $\mathcal M_\ell$ on a closed spin
manifold of dimension $n=7$ or $8$, where $\ell\geq1$. The connection
is induced by the spin connection, possibly tensored with a flat
real line bundle. If the generalized gradient corresponding to the
Cartan product vanishes, then, for $A=-\mathcal R_g$,
\begin{equation}\label{eq:spin-est-gradient-sign}
\int_M\langle[Q+t(K-S/8)]m,m\rangle\leq0
\end{equation}
for every
\begin{equation}\label{eq:spin-est-gradient-range}
-\frac13\leq t\leq u_{n,\ell},\qquad
u_{7,\ell}=\frac{3(2\ell+5)}{2\ell+15},\qquad
u_{8,\ell}=\frac{7(\ell+3)}{3(\ell+7)}.
\end{equation}
In particular,
\begin{equation}\label{eq:spin-est-weitzenbock-operator}
\mathcal W_\ell:=\frac12(Q+K^+-S/8)
 =\frac{N_\ell+6}{2N_\ell}
       \left[Q+\frac{E}{N_\ell+6}(K-S/8)\right]
\end{equation}
has nonpositive integral on this kernel.
\end{lemma}
\begin{proof}
We use the tensor-product decomposition and conformal-weight
conventions of \cite[Sections~1--2]{CGH00}. The highest weights of
$\mathcal M_\ell$ are
$$
(\ell+1/2,1/2,1/2)\quad\text{for }\Spin(7),\qquad
(\ell+1/2,1/2,1/2,\pm1/2)\quad\text{for }\Spin(8).
$$
Tensoring with the vector
representation gives 4 pairwise nonisomorphic summands, each with
multiplicity 1: the Cartan product, the summand obtained by
increasing the 2nd weight, the monogenic summand of degree $\ell$,
and the monogenic summand of degree $\ell-1$. In dimension 8 the 3rd
summand has the opposite chirality. Denote the corresponding
orthogonal generalized gradients by $G_+,G_h,G_0,G_-$ in this order,
with labels $\nu\in\{+,h,0,-\}$. Their
conformal weights are
\begin{equation}\label{eq:spin-est-conformal-weights}
b_+=\ell+\tfrac12,\qquad b_h=-\tfrac12,\qquad
b_0=-\tfrac{n-1}{2},\qquad b_-=-\ell-n+\tfrac32.
\end{equation}
Indeed, the selection rule of \cite[Section~2]{CGH00} gives precisely
these dominant weights. If $\Lambda$ is the source highest weight and
$\Lambda_\nu$ a target highest weight, the formula
$$
\operatorname{Cas}(\Lambda)
 =\sum_{i=1}^{\lfloor n/2\rfloor}\Lambda_i(\Lambda_i+n-2i),\qquad
b_\nu=\tfrac12\bigl(
\operatorname{Cas}(\Lambda_\nu)-\operatorname{Cas}(\Lambda)-(n-1)
\bigr)
$$
gives the values in \eqref{eq:spin-est-conformal-weights}.

Let $\Sym_0^\ell T^*M$ denote the bundle of trace-free symmetric
covariant $\ell$-tensors, and let $\mathcal D$ be the twisted Dirac
operator on $\Sigma\otimes\Sym_0^\ell T^*M$, restricted to sections of
$\mathcal M_\ell$. Only its domain is restricted; the output is not
projected back to $\mathcal M_\ell$. Its symbol is 0 on the 1st 2 summands, and
its squared multipliers on the last 2 are
\cite[Proposition~2.4]{HommaTomihisa21}
\begin{equation}\label{eq:spin-est-dirac-multipliers}
d_0=n-\frac{4\ell}{E},\qquad
 d_- =\frac{4(n+\ell-2)}E.
\end{equation}
To verify the normalization, recall that the orthogonal projection
from harmonic spinor polynomials of degree $\ell$ onto
$\mathcal M_\ell$ is
$$
\Pi_\ell=\Id+\frac{c(y)D_y}{E}.
$$
The adjoint of the Clifford contraction is
$f\mapsto(-\Pi_\ell(c_if))_i$. If $D_yf=0$, then
$$
\Pi_\ell(c_if)=c_if-\frac2E c(y)\partial_if,
$$
and contracting once more gives $(n-4\ell/E)f$. If
$f=c(y)g$, where $D_yg=0$ and $\deg g=\ell-1$, then
$$
D_y(c_ic(y)g)=(E-2)c_ig-2c(y)\partial_ig.
$$
Using $\sum_ic_ic(y)c_i=(n-2)c(y)$ gives
$$
\left[n-\frac{(n-2)(E-2)}E\right]f
 =\frac{4(n+\ell-2)}E f.
$$
This proves the multipliers \eqref{eq:spin-est-dirac-multipliers} for the orthogonal gradients.

The universal Weitzenb\"ock formula
\cite[Proposition~3.6]{SemmelmannWeingart10}, with $A=-\mathcal R_g$,
gives
\begin{equation}\label{eq:spin-est-universal-weitzenbock}
\int_M\langle Qm,m\rangle
 =\sum_{\nu\in\{+,h,0,-\}}b_\nu\|G_\nu m\|^2.
\end{equation}
Here is the sign check in our conventions. The conformal-weight
operator has blocks $B_{ij}=\rho(e_i\wedge e_j)$. The curvature of
the induced connection on $\mathcal M_\ell$ has components
$R^{\mathcal M}_{ij}=\rho(A(e_i\wedge e_j))$. Thus integration by
parts gives
$$
\int_M\langle B\nabla m,\nabla m\rangle
 =-\frac12\sum_{i,j}\int_M
       \langle m,\rho(e_i\wedge e_j)R^{\mathcal M}_{ij}m\rangle
 =\int_M\langle Qm,m\rangle.
$$
The twisted Schr\"odinger--Lichnerowicz formula
\cite[Chapter~II, Theorem~8.17]{LawsonMichelsohn89} gives the 2nd identity
\begin{equation}\label{eq:spin-est-twisted-lichnerowicz}
\int_M\langle(C+S/4)m,m\rangle
 =\|\nabla m\|^2-\|\mathcal Dm\|^2.
\end{equation}
Indeed, the spinor term in the square of $\mathcal D$ is
$\scal_g\Id/4=-S/4$, and the twisting term is $-C$.
Since $C+S/4=Q-K+S/8$, combining \eqref{eq:spin-est-universal-weitzenbock} and
\eqref{eq:spin-est-twisted-lichnerowicz} gives
$$
\int_M\langle[Q+t(K-S/8)]m,m\rangle
 =\sum_\nu\bigl[(1+t)b_\nu-t+td_\nu\bigr]\|G_\nu m\|^2,
$$
where $d_+=d_h=0$. The Cartan term vanishes by hypothesis. The
remaining coefficients are
$$
\begin{aligned}
h_h&=-\tfrac12-\tfrac32t,\\
h_0&=-\frac{n-1}{2}
       +t\left(\frac{n-1}{2}-\frac{4\ell}{E}\right),\\
h_-&=-\ell-n+\tfrac32
       +t\left(-\ell-n+\tfrac12+
                          \frac{4(n+\ell-2)}E\right).
\end{aligned}
$$
The 1st 2 are nonpositive exactly on the interval \eqref{eq:spin-est-gradient-range}. The
coefficient of $t$ in $h_-$ is negative, and
$$
h_-\big|_{t=-1/3}
 =-\frac{2(\ell+n-2)}3
  -\frac{4(\ell+n-2)}{3E}<0.
$$
This proves \eqref{eq:spin-est-gradient-sign}. The formula
\eqref{eq:spin-est-weitzenbock-operator} for $\mathcal W_\ell$ follows from
\eqref{eq:spin-est-compression} in Lemma~\ref{spin:compression}, and its
parameter $E/(N_\ell+6)$ lies in the interior of the admissible interval.
\end{proof}

\begin{theorem}\label{spin:first}
If $n=8$ and $\delta\leq\Sec_A\leq1$, with $\delta>1/4$, then
on $\mathcal M_1$,
\begin{equation}\label{eq:spin-est-linear8}
Q_1+\tfrac12K_1-S/16\geq5(4\delta-1)\Id.
\end{equation}
If $n=7$ and $1/4<\Sec_A\leq1$, then
\begin{equation}\label{eq:spin-est-linear7}
\mathcal W_1\geq\tfrac7{16}\Id,\qquad
Q_1+\tfrac7{15}K_1-\tfrac7{120}S\geq\tfrac{21}{40}\Id.
\end{equation}
\end{theorem}
\begin{proof}
Take $c=8$ in Lemma~\ref{spin:compression}, using
\eqref{eq:spin-est-compression} and \eqref{eq:spin-est-compression-ricci}. In dimension 8,
$$
I_8=16K^+-8K-S
 =\frac{48}{5}(Q+\tfrac12K-S/16),
$$
while
$$
I_8\geq8[2(8\delta-7)+8(1+\delta)]\Id
       =48(4\delta-1)\Id.
$$
In dimension 7, the same identities give
$$
I_8=15K^+-7K-S=12\mathcal W_1\geq\frac{21}{4}\Id,
$$
and
$$
\frac65\mathcal W_1=Q_1+\frac7{15}K_1-\frac7{120}S.
$$
These calculations apply to every monogenic linear spinor polynomial.
\end{proof}

\begin{lemma}\label{spin:slices}
Under $1/4<\Sec_A\leq1$, define
$$
\begin{array}{c|c|c}
n&H_\ell&B_\ell\\\hline
8&C+\frac32R_\ell+\frac\ell{16}S
     &R_\ell-\frac\ell{26}S\\[3pt]
7&C+\frac{22}{15}R_\ell+\frac\ell{15}S
     &R_\ell-\frac\ell{22}S.
\end{array}
$$
Then $B_\ell\geq0$, and
\begin{equation}\label{eq:spin-est-slice-bounds}
H_\ell\geq0\quad(n=8),\qquad
H_\ell\geq\frac{21\ell}{40}\Id\quad(n=7).
\end{equation}
In dimension 8, if $\delta\leq\Sec_A\leq1$ with $\delta>1/4$,
then $H_\ell\geq5\ell(4\delta-1)\Id$.
\end{lemma}
\begin{proof}
Fix all but 1 input of the symmetric coefficient tensor of a
monogenic polynomial. The remaining vector-spinor is Clifford-traceless,
meaning that its Clifford contraction vanishes by monogenicity.
Since $K_1$ is the
compression of $\Ric_A\otimes\Id_\Sigma$ to $\mathcal M_1$,
Theorem~\ref{spin:first}, through \eqref{eq:spin-est-linear8} and
\eqref{eq:spin-est-linear7}, gives
$$
C_1+\tfrac32K_1+S/16\geq5(4\delta-1)\Id
                         \quad(n=8),
$$
$$
C_1+\tfrac{22}{15}K_1+S/15\geq21\Id/40
                         \quad(n=7).
$$
Sum over all coefficient indices. The mixed curvature term is
additive in these indices, so its sum is $C$; the Ricci terms sum to
$R_\ell$, and the scalar term and lower bound are each counted
$\ell$ times. This proves the estimates \eqref{eq:spin-est-slice-bounds} for $H_\ell$ in the
coefficient norm, hence also in the sphere norm.

For a unit vector $v$, put $r=\Ric_A(v,v)$. Then
$r\geq(n-1)/4$, whereas all sectional curvatures of planes in
$v^\perp$ are at most 1. Therefore
$$
\scal_A\leq2r+(n-1)(n-2)\leq(4n-6)r.
$$
It follows that $\Ric_A\geq\scal_A\Id/(4n-6)$, and summing this
inequality over the coefficient indices proves $B_\ell\geq0$.
\end{proof}

\subsection{Vanishing in degrees at least 2}

The following 2 propositions use positive linear combinations of
Lemmas~\ref{spin:compression} and~\ref{spin:slices} for which the Ricci term cancels.
The resulting endomorphism is one of those in
Lemma~\ref{spin:gradients}.

\begin{proposition}\label{spin:finite8}
Let $(M^8,g)$ be a closed strictly quarter-pinched spin manifold. For
$2\leq\ell\leq20$, the kernel of
the generalized gradient corresponding to the Cartan product on
$\mathcal M_\ell$ is 0.
\end{proposition}
\begin{proof}
First work at a point with $1/4<\Sec_A\leq1$, and set
$$
c_0=\frac{-\ell^3+9\ell^2+144\ell-32}{3(\ell+4)}.
$$
Choose
\begin{equation}\label{eq:spin-est-finite8-parameters}
(c,\nu)=
\begin{cases}
(c_0,0),&2\leq\ell\leq13,\\[2pt]
\displaystyle\left(24,
 \frac{78(\ell+4)(24-c_0)}{25\ell-52}\right),&14\leq\ell\leq20,
\end{cases}
\qquad
\mu=\frac43L-\frac23\nu.
\end{equation}
All coefficients are nonnegative, and $c,\mu>0$. For the 1st
range this follows from
$144+9\ell-\ell^2\geq92$. For the 2nd range, the explicit values
are
$$
\mu=\frac{16(\ell^3+20\ell^2+31\ell-312)}{25\ell-52},\qquad
\nu=\frac{26(\ell^3-9\ell^2-72\ell+320)}{25\ell-52},
$$
whose numerators are positive at $\ell=14$ and increasing thereafter.

The identities of Lemma~\ref{spin:compression} give
$$
N_\ell I_c=aK+6(c+D)C-2LR_\ell+(c-\ell^2-8\ell+7)S,
\qquad a=c(E+6)+4D.
$$
Set
$$
b=6(c+D)+\mu,\qquad
 d=c-\ell^2-8\ell+7+\mu\ell/16-\nu\ell/26.
$$
Our choices \eqref{eq:spin-est-finite8-parameters} satisfy $-2L+3\mu/2+\nu=0$ and $4d=b-a/2$.
Consequently
\begin{equation}\label{eq:spin-est-finite8-combination}
N_\ell I_c+\mu H_\ell+\nu B_\ell
 =aK+bC+dS
 =b[Q+t(K-S/8)],\qquad t=\frac{a-b}{b}.
\end{equation}
The pointwise lower bound from \eqref{eq:spin-est-compression-quarter} is $N_\ell j_\ell\Id$, where
\begin{equation}\label{eq:spin-est-finite8-lower}
j_\ell=\frac{\ell+7}{4}[c(\ell+1)-23\ell+7].
\end{equation}
For completeness, the following table records the minima over each
finite range, obtained by substitution in \eqref{eq:spin-est-finite8-parameters},
\eqref{eq:spin-est-finite8-combination}, and \eqref{eq:spin-est-finite8-lower}:
\begin{center}
\renewcommand{\arraystretch}{1.35}
\begin{tabular}{c|c|c|c}
range of $\ell$ & $\min N_\ell j_\ell$ & $\min t$
  & $\min(u_{8,\ell}-t)$\\\hline
$2\leq\ell\leq13$ & $225$ & $\dfrac{251}{420}$
  & $\dfrac{1201}{3204}$\\[5pt]
$14\leq\ell\leq20$ & $8505$ & $\dfrac{1091}{2073}$
  & $\dfrac{58748}{47259}$
\end{tabular}
\end{center}
Thus $b>0$, $N_\ell j_\ell>0$, and $0<t<u_{8,\ell}$.

For the original metric, write $A=\kappa\widehat A$, where
$\kappa>0$ is the uniform curvature scale in the strict pinching
assumption, and apply the normalized estimate to $\widehat A$.
Linearity in the curvature tensor gives the lower bound
$\kappa N_\ell j_\ell\Id$ for $A$.
If the Cartan generalized gradient of $m$ vanishes,
\eqref{eq:spin-est-gradient-sign} in Lemma~\ref{spin:gradients} gives the opposite integrated sign. Since
$\kappa>0$, this forces $m=0$.
\end{proof}

\begin{proposition}\label{spin:finite7}
Let $(M^7,g)$ be a closed strictly quarter-pinched spin manifold. For
$2\leq\ell\leq15$, the kernel of
the generalized gradient corresponding to the Cartan product on
$\mathcal M_\ell$ is 0.
\end{proposition}
\begin{proof}
Normalize the curvature by the uniform scale as in the preceding proof, and put
\begin{equation}\label{eq:spin-est-finite7-parameters}
\begin{aligned}
c_0&=8-\frac{2L(4\ell-59)}{11N_\ell},\\
(c,\nu)&=
\begin{cases}
(c_0,0),&2\leq\ell\leq13,\\[2pt]
\displaystyle\left(20,
 \frac{11N_\ell(20-c_0)}{8\ell-15}\right),&\ell=14,15,
\end{cases}\\
\mu&=\frac{15}{22}(2L-\nu).
\end{aligned}
\end{equation}
For $2\leq\ell\leq13$, we have $c>8$ and $\mu=15L/11>0$.
For $\ell=14,15$, respectively,
$$
(\mu,\nu)=\left(\frac{374850}{1067},\frac{2982}{97}\right),\qquad
(\mu,\nu)=\left(\frac{2690}{7},\frac{1100}{21}\right).
$$
In particular, all required multipliers are nonnegative.

This time Lemma~\ref{spin:compression} gives
$$
N_\ell I_c=aK+6(c+D)C-2LR_\ell+(c-\ell^2-7\ell+6)S,
\qquad a=c(E+6)+4D.
$$
Set
$$
b=6(c+D)+\mu,\qquad
 d=c-\ell^2-7\ell+6+\mu\ell/15-\nu\ell/22.
$$
Using \eqref{eq:spin-est-finite7-parameters}, the coefficients satisfy
$-2L+22\mu/15+\nu=0$ and $4d=b-a/2$. Hence
\begin{equation}\label{eq:spin-est-finite7-combination}
N_\ell I_c+\mu H_\ell+\nu B_\ell
 =b[Q+t(K-S/8)],\qquad t=\frac{a-b}{b}.
\end{equation}
Its pointwise lower bound is $p_\ell\Id$, where
\begin{equation}\label{eq:spin-est-finite7-lower}
p_\ell=\frac{N_\ell D}{4}[c(\ell+1)-19\ell+6]
                      +\frac{21\ell\mu}{40}.
\end{equation}
As above, substitution in \eqref{eq:spin-est-finite7-parameters},
\eqref{eq:spin-est-finite7-combination}, and \eqref{eq:spin-est-finite7-lower} gives the following exact minima:
\begin{center}
\renewcommand{\arraystretch}{1.35}
\begin{tabular}{c|c|c|c}
range of $\ell$ & $\min p_\ell$ & $\min t$
  & $\min(u_{7,\ell}-t)$\\\hline
$2\leq\ell\leq13$ & $\dfrac{14855}{44}$
  & $\dfrac{10265}{32961}$ & $\dfrac{19760}{29463}$\\[5pt]
$14\leq\ell\leq15$ & $\dfrac{20448295}{2134}$
  & $\dfrac{479}{1103}$ & $\dfrac{5013440}{2712999}$
\end{tabular}
\end{center}
Thus $p_\ell>0$ and $0<t<u_{7,\ell}$. Restoring the uniform curvature scale
and applying \eqref{eq:spin-est-gradient-sign} in Lemma~\ref{spin:gradients} proves the result.
\end{proof}

\begin{proposition}\label{spin:tails}
Under $1/4<\Sec_A\leq1$,
\begin{equation}\label{eq:spin-est-tail8}
\mathcal W_\ell\geq\frac{2\ell^2-42\ell+7}{8}\Id>0
                  \quad(n=8,\ \ell\geq21),
\end{equation}
\begin{equation}\label{eq:spin-est-tail7}
\mathcal W_\ell\geq\frac{20\ell^2-323\ell+60}{80}\Id>0
                  \quad(n=7,\ \ell\geq16).
\end{equation}
Consequently, on a closed negatively curved spin manifold with
strict quarter pinching, the kernel of the Cartan generalized
gradient in these degrees is 0.
\end{proposition}
\begin{proof}
Since $\mathcal W_\ell=(K+K^++C)/2$, Lemma~\ref{spin:slices} yields
$$
\mathcal W_\ell\geq\frac{K+K^+}{2}
                    -\frac34R_\ell-\frac\ell{32}S
                         \quad(n=8),
$$
$$
\mathcal W_\ell\geq\frac{K+K^+}{2}
                    -\frac{11}{15}R_\ell-\frac\ell{30}S
                    +\frac{21\ell}{80}\Id
                         \quad(n=7).
$$
Use the polynomial curvature bounds \eqref{eq:spin-est-polynomial-bounds} preceding
Lemma~\ref{spin:negative}, together with
$$
R_\ell\leq(n-1)\ell\Id,\qquad S\leq n(n-1)\Id.
$$
This gives the 2 polynomials in \eqref{eq:spin-est-tail8} and \eqref{eq:spin-est-tail7}. Their numerators are positive
at $\ell=21$ and $16$, respectively, and increasing thereafter.
The vanishing assertion follows by restoring the uniform curvature scale and
Lemma~\ref{spin:gradients}.
\end{proof}

\subsection{The lowest degrees and the exceptional structures}

\begin{lemma}\label{spin:variable}
Let $n>1$. Suppose a real unit spinor $s$ and a smooth real function
$\lambda$ satisfy
$\nabla_Ys=\lambda c(Y)s$ for every tangent vector $Y$. Then
\begin{equation}\label{eq:spin-est-variable-curvature}
d\lambda=0,\qquad\scal_g=4n(n-1)\lambda^2.
\end{equation}
In particular, such a spinor cannot exist on an open set with
negative scalar curvature.
\end{lemma}
\begin{proof}
Here $\mathcal D$ denotes the ordinary spin Dirac operator and
$\nabla^*\nabla$ the positive rough Laplacian. They give
$$
\mathcal Ds=-n\lambda s,\qquad
\mathcal D^2s=-n c(d\lambda)s+n^2\lambda^2s,
$$
$$
\nabla^*\nabla s=-c(d\lambda)s+n\lambda^2s.
$$
The Schr\"odinger--Lichnerowicz formula
\cite[Chapter~II, Theorem~8.8]{LawsonMichelsohn89} therefore gives
$$
-(n-1)c(d\lambda)s+
 \left[n(n-1)\lambda^2-\frac{\scal_g}{4}\right]s=0.
$$
Taking the real inner product with $s$ proves the scalar-curvature
identity in \eqref{eq:spin-est-variable-curvature}. The remaining equation is $c(d\lambda)s=0$, and
$|c(d\lambda)s|^2=|d\lambda|^2|s|^2$ proves $d\lambda=0$.
\end{proof}

\begin{theorem}\label{spin:n8}
Let $(M^8,g)$ be closed, connected, oriented and strictly
quarter-pinched. There is no odd
flow-invariant metric-compatible $G_2$-structure with finite Fourier
content on the normal bundle over $SM$.
\end{theorem}
\begin{proof}
Suppose such a structure exists. By Lemmas~\ref{spin:pair8} and~\ref{spin:top}, on a
connected finite cover its spinor factors have a nonzero highest
monogenic coefficient $m_\ell$ whose Cartan generalized gradient
vanishes. Propositions~\ref{spin:finite8} and~\ref{spin:tails} exclude every $\ell\geq2$.

For $\ell=1$, \eqref{eq:spin-est-linear8} in Theorem~\ref{spin:first} gives, in the normalized curvature,
$$
Q_1+\tfrac12(K_1-S/8)\geq5(4\delta-1)\Id.
$$
Here $\delta>1/4$, and the parameter $t=1/2$ lies in the interval \eqref{eq:spin-est-gradient-range} in
Lemma~\ref{spin:gradients}. After restoring the uniform curvature scale,
the endomorphism is positive everywhere. Its integral on $m_1$ is
nonpositive by \eqref{eq:spin-est-gradient-sign} in Lemma~\ref{spin:gradients}, so $m_1=0$, a contradiction.

For $\ell=0$, the smaller factor is a unit spinor on the base.
Its invariance says $\nabla_vs=0$ for every unit $v$, hence
$\nabla s=0$. The Schr\"odinger--Lichnerowicz formula then gives
$\scal_gs=0$, contrary to negative scalar curvature.
\end{proof}

\begin{theorem}\label{spin:n7}
Let $(M^7,g)$ be closed, connected, oriented and strictly
quarter-pinched.
There is no odd flow-invariant orthogonal almost complex structure
with finite Fourier content on the normal bundle over $SM$.
\end{theorem}
\begin{proof}
Suppose such an almost complex structure $J$ exists. By
Lemma~\ref{spin:spin7} and \eqref{eq:spin-lift-rank-two-projector}, on a finite spin cover it determines the even
rank-2 orthogonal projector
$$
P=(\Id-c(v)c(J))/4.
$$
By Theorem~\ref{spin:lift7}, after a further finite cover the spinor factors
are globally defined polynomial unit spinors, and their equations
have a smooth real coefficient $\lambda$ depending only on the
base point. Lemma~\ref{spin:top} supplies a nonzero highest monogenic
coefficient $m_\ell$ with vanishing Cartan generalized gradient for
the ordinary induced spin connection. The term involving $\lambda$
in the equation for the factor of larger degree has lower Fourier
degree and does not enter its highest-degree equation.

Propositions~\ref{spin:finite7} and~\ref{spin:tails} exclude $\ell\geq2$. For $\ell=1$,
the estimate $\mathcal W_1\geq7\Id/16$ in \eqref{eq:spin-est-linear7} of Theorem~\ref{spin:first}, restored
to the original curvature scale, contradicts
Lemma~\ref{spin:gradients}.

For $\ell=0$, the smaller factor is a base unit spinor satisfying
$$
\nabla_vs=\lambda(x)c(v)s
$$
for every unit $v$. By linearity this holds for all tangent vectors.
\eqref{eq:spin-est-variable-curvature} in Lemma~\ref{spin:variable} gives $\scal_g=168\lambda^2\geq0$, again a
contradiction.
\end{proof}

\section{Proof of the main result}
\label{main:assembly}

We complete the proof of Theorem~\ref{main:quarter}.

The topological reduction in dimension $134$ has one additional case.

\begin{lemma}\label{main:adjoint134}
The tangent bundle of $S^{133}$ does not admit a reduction through the
adjoint representation of $E_7/\Z_2$.
\end{lemma}
\begin{proof}
Let $G$ be a compact connected Lie group of dimension $d$. The adjoint
reduction theorem of Mart\'inez-Sandoval
\cite[Theorem~4.3]{MS26}, together with
Adams's theorem on vector fields on spheres
\cite[Theorem~1.1]{Adams62}, gives the
necessary condition
\begin{equation}\label{eq:main-adjoint-rank}
\rank G\leq\rho(d+1)
\end{equation}
for an adjoint reduction of $TS^d$. Indeed, every vector bundle over a
sphere associated through the adjoint representation admits at least
$\rank G-1$ independent sections, whereas $TS^d$ admits at most
$\rho(d+1)-1$. For $G=E_7/\Z_2$, we have $d=133$ and
$\rank G=7$, but $\rho(134)=2$. This is a contradiction to \eqref{eq:main-adjoint-rank}.
\end{proof}

\begin{proof}[Proof of Theorem~\ref{main:quarter}]
In dimension 2, the
oriented frame flow is the geodesic flow, as explained in
Section~\ref{main:reduction}. In odd dimensions other than 7, ergodicity
follows from \cite[Theorem~1.1]{BrinGromov80}. The 7-dimensional
case follows from Theorem~\ref{spin:n7}.

Assume now that $n$ is even and that the oriented frame flow is not
ergodic. By Proposition~\ref{main:alternatives}, a finite Riemannian cover carries
one of the indicated invariant structures. Choose a connected component
of this cover. It still covers $M$ and satisfies the same strict
quarter-pinching bounds.

If $n=8$, the even projector is excluded by
Theorem~\ref{pr:closure}, and the invariant $G_2$-structure on the
normal bundle is excluded by Theorem~\ref{spin:n8}. If $4\mid n$ and
$n\geq12$, Theorem~\ref{pr:closure} similarly excludes the proper
even invariant projector.

If $n=4$, or $n\equiv2\pmod4$ with $n\ne134$, there is an odd
invariant unit section $u$ of $N$ with finite Fourier content. Then
$
P=u\otimes u
$
is an even invariant rank-1 orthogonal projector on $N$, also with
finite Fourier content, contrary to Theorem~\ref{rank1:exclusion}. Finally,
when $n=134$, Lemma~\ref{main:adjoint134} excludes the adjoint reduction,
and the same rank-1 argument excludes the remaining alternative.
\end{proof}

\section*{Acknowledgements}

We are grateful to L. Mart\'inez-Sandoval for many helpful discussions. In particular, he points out that
the exceptional case in dimension $134$ can be treated using
his theory of adjoint reductions. His observation led to the
exclusion of the reduction of $TS^{133}$ through the adjoint
representation of $E_7/\mathbb Z_2$, as explained in
Lemma~\ref{main:adjoint134}, and thereby allowed us to include
dimension $134$ in our result.

\section*{Declaration of AI use}

The authors used ChatGPT 6 during the
development and revision of this work. This assistance included
the exploration of proof strategies and contributed to the
development of some of the new mathematical ideas incorporated
into the revised manuscript. In particular, the tool suggested an algebraic approach to
eliminating the quartic case for even polynomial normal orthogonal
projectors, based on the parity of the minimal power of
$q(y)=|y|^2$ needed to clear the denominators of their homogeneous
rational extensions. This suggestion contributed to the
degree-parity argument in Theorem~\ref{pr:parity}, which is one of the key ingredients in the strict
quarter-pinching argument for invariant normal projectors. ChatGPT 6 was also used in polishing the text. The authors take full responsibility
for all mathematical arguments, references, and conclusions
presented in this paper.

\end{document}